\documentclass[onefignum,onetabnum]{siamart171218}

\usepackage{lipsum}
\usepackage{amsmath,bm}
\usepackage{amsfonts}
\usepackage{graphicx}
\usepackage{wrapfig}
\usepackage{psfrag}
\usepackage{tikz}
\usetikzlibrary{calc,positioning}
\definecolor{corange}{rgb}{0.93, 0.57, 0.13}
\usepackage{simplewick}
\usepackage{mathrsfs}
\usepackage{url,hyperref}
\usepackage{setspace}
\usepackage{pifont}
\usepackage{wrapfig}

\usepackage{algorithm}
\usepackage{algorithmic}
\usepackage{amssymb}
\usepackage{booktabs}
\usepackage{comment}
\usepackage{subfigure}
\usepackage{float}
\usepackage{cite}
\usepackage{epstopdf}
\newtheorem{exa}{\bf Example}
\usepackage{tikz}

\newcommand{\bs}[1]{\boldsymbol{#1}}
\allowdisplaybreaks[4]

\newtheorem{lem}{Lemma}[section]

\newtheorem{rem}{Remark}[section]

\def \bc{\bs c}

\def \bx{\bs x}

\def \by{\bs y}
\def \bz{\bs z}

\def\B{\mathbb{B}}

\def\T{\mathcal{T}}

\def\d{\mathrm{d}}

\newcommand \dint {\displaystyle\int}

\usepackage{lipsum}
\usepackage{amsfonts}
\usepackage{graphicx}
\usepackage{epstopdf}
\ifpdf
\DeclareGraphicsExtensions{.eps,.pdf,.png,.jpg}
\else
\DeclareGraphicsExtensions{.eps}
\fi

\newsiamremark{remark}{Remark}
\newsiamremark{hypothesis}{Hypothesis}
\crefname{hypothesis}{Hypothesis}{Hypotheses}
\newsiamthm{claim}{Claim}

\title{An $\MakeLowercase{hp}$-version time stepping spectral Monte Carlo method for semi-linear parabolic equations \thanks{Submitted to the editors DATE.
\funding{The research of the first author is partially supported by the Fundamental Research Funds for the Central Universities (No. CXJJ-2024-437). The research of the second author is partially supported by the Fundamental Research Funds for the Central Universities (No. CXJJ-2025-451). The research of the third author is partially supported by the National Natural Science Foundation of China (Nos. 12571389 and 12271365). ${}^{\star}$Corresponding author. }}}

\author{Jiaying Feng\thanks{School of Mathematics, Shanghai University of Finance and Economics, Shanghai 200433, China. Email: \email{fengjiaying@163.sufe.edu.cn} (J. Feng);
\email{2023213122@stu.sufe.edu.cn} (Z. Hui); \email{ctsheng@sufe.edu.cn} (C. Sheng); \email{xu.chenglong@shufe.edu.cn} (C. Xu).  }
\and  Zhiyuan Hui$^{\dagger}$, Changtao Sheng$^{\dagger,\star}$\and Chenglong Xu$^{\dagger}$ }
\ifpdf
\hypersetup{
pdftitle={An $hp$-version time stepping spectral Monte Carlo method for semi-linear PDEs: Long time simulation and initial simgularity},
pdfauthor={J. Feng, Z. Hui, C. Sheng, and  C. Xu}
}
\begin{document}
% \nolinenumbers
\maketitle

% REQUIRED
\begin{abstract}
In this paper, we present an $hp$-version time-stepping spectral Monte Carlo method for solving semi-linear parabolic equations. The key innovation lies in constructing an exponentially accurate stochastic algorithm that integrates a residual iteration scheme on Gauss-type nodes in both temporal and spatial directions with a reconstruction strategy rooted in spectral methods. To address the long-time simulations and initial singularities that are often challenging for traditional stochastic algorithms (e.g., walk-on-spheres method), we further develop an $hp$-version time-stepping framework that employs multiple time steps and, respectively, geometric time partitions with linearly increasing polynomial degrees to handle these difficulties. Notably, the proposed algorithm bypasses the need to solve linear systems required by traditional spectral methods and remarkably supports parallel computation at both temporal and spatial grid points. We rigorously establish exponential convergence rates for the multistep method within a finite number of iterations. Extensive numerical experiments are conducted to demonstrate the spectral accuracy and computational efficiency of the proposed method in long-time simulations, problems with initial singularities, and a five-dimensional problem, thereby validating the theoretical results.

\end{abstract}

% REQUIRED

% REQUIRED

\section{Introduction}
Stochastic methods have been widely employed across various scientific and engineering disciplines due to their distinct advantages in simulating high-dimensional problems, handling partial differential equations (PDEs) on complex geometric domains, supporting parallel computation, and offering relatively simple implementation. Applications include, but are not limited to, molecular dynamics simulations, radiation transport, uncertainty quantification, quantum mechanics, risk management,  and option pricing (see, e.g., \cite{Cetinkaya2019,Cai2011,Bouchaud2000,Glasserman2004,LeMaetre2010,Yong1999,Shao2020,Lei2025} and references therein). However, such methods typically achieve only low-order convergence rates, often limited to $O(M^{-\frac{1}{2}})$ or $O(M^{-1})$, where $M$ denotes the number of samples. As a result, obtaining stable and reliable numerical solutions generally requires a large number of samples, which in turn leads to substantial computational cost and time consumption during simulation or sampling.

In view of these limitations, extensive research has focused on improving the sampling efficiency and/or accuracy of traditional Monte Carlo methods, for example through techniques such as importance sampling and control variates (cf. \cite{Gelman1998,Glasserman2004,Gobet2005,Needell2016}). Against this background, a variety of stochastic algorithms with enhanced performance have been developed, including the Quasi-Monte Carlo method \cite{Caflisch1998,Dick2013,Kuo2012,Graham2015},  Multilevel Monte Carlo method~\cite{Charrier2013,E2019,Giles2008,Giles2015},  hybrid Monte Carlo-Picard iteration schemes~\cite{E2019,GobetLemor2005,Hutzenthaler2020,Ramirez2006}, Markov Chain Monte Carlo methods~\cite{Andrieu2010,Dodwell2015,Green1995,Hastings1970}, and Random batch methods~\cite{Jin2020,Jin2021,Ko2021,Li2020}, among others. Moreover, the integration of stochastic approaches with neural networks has gradually become a prominent research direction in this field (see, e.g.,~\cite{Beck2019,E2017,Fresca2021,Takahashi2022,Cai2025soc} and the references therein).

Attaining higher-order accuracy, and even approaching machine precision, has long been a central objective in the numerical analysis community. Gobet and Maire\\ ~\cite{Gobet2004} first proposed the Spectral Monte Carlo (SMC) method for the Poisson equation, which combines variance reduction techniques from sequential Monte Carlo with spectral collocation, achieving spectral accuracy within a finite number of iterations. They later enhanced this approach by incorporating sequential control variates and proved geometric convergence in both bias and variance~\cite{Gobet2005}. Subsequently, this method and its variants have been further extended to a broad range of problems (cf.~\cite{GobetMenozzi2010,Gobet2010,Maire2015,Billaud-Friess2024}).  
More recently, Feng et al.~\cite{Feng2025} developed a spectral Monte Carlo method that integrates residual iteration techniques, generalized Jacobi function reconstruction, space-time spectral methods, and the walk-on-spheres method, thereby extending the framework to both linear elliptic and parabolic equations and achieving spectral accuracy for both integer-order and fractional-order problems.  For the convenience of the reader, we briefly describe the algorithm therein.  
 For instance, for a linear parabolic equation $\partial_t u+\mathcal{L}u = f$, where $\mathcal{L}$ denotes a linear differential/integral operator, the residual $\varepsilon^{(k)} = u - u^{(k-1)}$ is obtained by solving the equation 
\begin{equation}\label{iterlinear}\partial_t \varepsilon^{(k)}+\mathcal{L}\varepsilon^{(k)} = f - \partial_t u^{(k-1)}-\mathcal{L}u^{(k-1)},\;\;\;k=0,1,2,\cdots\end{equation}
using stochastic methods at Gauss-type points, followed by the update $u^{(k)}= u^{(k-1)} + \varepsilon^{(k)}$.
It is worth noting that another key aspect of spectral Monte Carlo method lies in exploiting the analytical derivative relationships of orthogonal polynomial/functions, which enables accurate reconstruction of both the solution and its derivatives from coarse data during the iterative process, thereby achieving spectral accuracy.

To the best of our knowledge, several issues in the development of stochastic algorithms remain unsolved:  
\begin{itemize}
\item[\ding{172}]developing the spectral Monte Carlo method to semi-linear parabolic equations;  
\item[\ding{173}] addressing numerical instabilities that arise in traditional stochastic algorithms (e.g., walk-on-spheres method (WoS)) during long-time simulations;
\item[\ding{174}] achieving high accuracy for time-dependent problems with initial singularity.
\end{itemize}
Hence, the primary objective of this paper is to develop a novel stochastic algorithm capable of overcoming the aforementioned three issues, aimed at solving the following semilinear parabolic equation:
\begin{equation}\label{mainproblem}
\begin{cases}
\partial_t u(\bx,t)+\mathcal{L} u(\bx,t) = f\big(u(\bx,t)\big), & (\bx,t) \in \Omega \times (0,\infty), \\[4pt]
u(\bx,t)=g(\bx,t), & (\bx,t) \in \partial \Omega \times (0,\infty), \\[4pt]
u(\bx,0)=u_0(\bx), & \bx \in \Omega,
\end{cases}
\end{equation}
where $\Omega$ is an open bounded domain, $\mathcal{L}$ denotes a positive linear differential operator acting on the spatial variables, and $f(u)$ represents a semilinear or quasilinear operator in $u$, which may involve lower-order derivatives. To address issue $\text{\ding{172}}$, we formulate the iterative scheme by introducing the residual function $\varepsilon^{(k)} = u^{(k)} - u^{(k-1)}$, where $u^{(k)}$ denotes the $k$-th iterative solution, and then solving
\begin{equation}\label{iterscheme}
\partial_t \varepsilon^{(k)} +\mathcal{L} \varepsilon^{(k)} = f(u^{(k)}) - \partial_t u^{(k-1)}- \mathcal{L} u^{(k-1)},
\end{equation}
followed by the update $u^{(k)}= u^{(k-1)}+ \varepsilon^{(k)}$.
Again, the core idea of the SMC is to compute the values of the residual $\varepsilon^{(k)}$ at Gauss-type points using Monte Carlo algorithms. By combining these values with space-time spectral interpolation and reconstruction techniques for $ \partial_t u^{(k-1)}$ and $\mathcal{L} u^{(k-1)}$, the desired spectral accuracy can be achieved. Note that, when $f$ does not depend nonlinearly on $u$, the above formulation naturally reduces to the scheme~\eqref{iterlinear}.
% By transferring the linear terms from the right-hand side of~\eqref{iterscheme} to the left, the residual scheme reduces precisely to a Picard iteration scheme, whose derivation will be provided in Section~\ref{algormxt}.
 Moreover, by introducing  $\varepsilon^{(k)} = u^{(k)} - u^{(k-1)}$, the residual equation~\eqref{iterscheme} can be recast as a Picard-type iteration, whose derivation is presented in Section~\ref{algormxt}. We emphasize that, in contrast to the linear PDE case, the iterative scheme developed here for semilinear parabolic equations admits a more natural formulation while incurring no additional computational cost. 
The key advantages and contributions of this work are summarized as follows:
\begin{itemize}
\item {\bf Spectral accuracy and parallel computation}: The proposed method not only retains the spectral accuracy of traditional spectral methods, but also possesses a key distinguishing feature: due to the stochastic nature of the algorithm, the numerical solutions at each temporal and spatial grid point can be computed in parallel during each iteration, without the need to solve linear systems. This feature fundamentally differentiates from the classical space-time spectral methods, as such methods for nonlinear complex systems often struggle to obtain stable and efficient parallel solvers due to the highly ill-conditioned asymmetric mass or stiffness matrices, even in the linear case (cf. \cite{Shen2019,Kong2024}). 

\item {\bf Long time simulation and initial singularities}: To address the challenges faced by traditional stochastic algorithms, particularly the instability in long-time simulations and the difficulty in accurately resolving initial singularities (i.e., issues $\text{\ding{173}}$-$\text{\ding{174}}$), we further develop an $hp$-version time-stepping scheme. More precisely, similar to a time-marching scheme, we solve the following subproblem  on each subinterval $(t_{n-1},t_n]$ :
\begin{equation}\label{multischeme}
\begin{cases}
\partial_t u^n(\bx,t)+\mathcal{L} u^n(\bx,t) = f\big(u^{n}(\bx,t)\big), & (\bx,t) \in \Omega \times (t_{n-1},t_n], \\[4pt]
u^n(\bx,t)=g(\bx,t), & (\bx,t) \in \partial \Omega \times (t_{n-1},t_n], \\[4pt]
u^n(\bx,t_{n-1})=u_\ast^{n-1}(\bx,t_{n-1}), & \bx \in \Omega,
\end{cases}
\end{equation}
where $u^n$ and $u^n_\ast$ denote the exact solution and the numerical solution of~\eqref{mainproblem} on the $n$-th element, respectively.
On each subinterval, the iterative scheme\\ ~\eqref{iterscheme} is constructed and solved sequentially.
This strategy provides greater flexibility, as it allows the use of large time steps, geometric time partitioning with linearly increasing polynomial degrees, and locally refined meshes to mitigate the issues caused by excessively high polynomial degrees, thereby effectively overcoming these difficulties.

\item {\bf Error estimate}: As far as we know, this work is the first to explore Monte Carlo methods with spectral accuracy in both space and time for solving nonlinear equations, and it further establishes a rigorous convergence analysis for the proposed methods. The methodology introduced in this paper not only provides exponentially accurate numerical solutions for one-dimensional nonlinear parabolic equations, but also exhibits versatility in extending to multidimensional problems and irregular domains (see Examples~4 and~5), and can even be integrated with neural networks. 
\end{itemize}

The rest of the paper is organized as follows. In Section \ref{sect2}, we develop the time-stepping spectral Monte Carlo method for solving the semi-linear parabolic equation and provide detailed implementation of the algorithm. Rigorous error estimates for the proposed method are presented in Section \ref{sect3}. In Section\;\ref{sect4}, several numerical examples are given to demonstrate the accuracy and efficiency of the proposed spectral Monte Carlo method for nonlinear PDEs. Finally, conclusions are drawn in the last section.

\section{$hp$-version of time-stepping spectral Monte Carlo method} \label{sect2}
In this section, we propose an exponentially accurate time-stepping spectral Monte Carlo algorithm for approximating solutions of semi-linear parabolic equations \eqref{mainproblem}. In practice, for long-term simulations and problems with initial singularities, where single-step methods may be ineffective, multi-step extensions are preferable.
\subsection{Preliminaries}
We begin with some notation. For convenience, let $M$ denote the number of sample paths, $\varepsilon_0$ the prescribed tolerance of the expected error, $N_x$ the number of spatial nodes, $M_n$ the number of temporal nodes on each subinterval, and $k_0$ the maximum number of iterations. Let $\T_h$ be a mesh of the time interval $I := [0, T]$, defined as
\begin{equation}\label{mesh}
\T_h:=\{t_n: 0 =t_0<t_1<\dots <t_{N_h}=T\}.
\end{equation}
We denote $h_n :=t_n-t_{n-1}$, $\T_n =(t_{n-1},t_{n}]$ and $u^n(x,t)$ the solution of \eqref{ufg_0} on the $n$-th element, that is,
\begin{equation*}
u^n(\bx,t) = u(x,t),\;\;\;\forall(x,t)\in \Omega \times \T_n,\;\;\;1\leq n \leq {N_h}.
\end{equation*}
Denote by $\{ t_{n,\ell} \}_{\ell = 0}^{M_n}$ the shifted Legendre-Gauss-Lobatto quadrature nodes on the interval $\T_n$. 
To evaluate the solution at these nodes, we simulate the numerical solution by taking $t_{n-1}$ as the initial time. 
To this end, we further subdivide the interval $(t_{n-1}, t_{n,\ell}]$ into $N_\ast$ uniform subintervals:
\begin{equation}\label{time division}
\T_{n,\ell}:=\big\{t_{n-1}=t_{n,\ell,0}<t_{n,\ell,1}<t_{n,\ell,2}<\cdots<t_{n,\ell,N_\ast-1}<t_{n,\ell,N_\ast} = t_{n,\ell}\big\},
\end{equation}
where $\Delta_{n,\ell} = \frac{t_{n,\ell}-t_{n-1}}{N_\ast}$ with $1 \leq q \leq N_\ast$, and $N_\ast$ is a relatively small integer that typically depends on the interval size $h_n$ and the index $\ell$.

Moreover, we define  
$\B^{d}_{r}(\bc_{0}) = \{\bx \in \mathbb{R}^{d} : |\bx - \bc_{0}| \leq r \}$  
to be a ball of radius $r > 0$ centered at $\bc_{0} \in \mathbb{R}^{d}$. For simplicity, we denote $\B_{r}^d = \B_{r}^{d}(\bm{0})$.
We then define the exit time of the process $X_t$ from the ball $\mathbb{B}_r^d(X_{t_{n,\ell,q-1}})$ as  
$\tau_q := \inf \{\, s \mid X_s \notin \mathbb{B}_r^d(X_{t_{n,\ell,q-1}}) \}, \; 1 \leq q \leq N_\ast.$  Within each time interval $[t_{n,\ell,q-1},\, t_{n,\ell,q})$, the process $X_t$ starts from the location $X_{t_{n,\ell,q-1}}$ and leaves the ball $\mathbb{B}_r^d(X_{t_{n,\ell,q-1}})$ at time $t_{n,\ell,q}$.  
This implies that we can directly set $\tau_q \approx \Delta_{n,\ell}$ (cf. \cite{Sheng2025}).

 % for any $h \rightarrow 0$. 

\subsection{Stochastic algorithm for parabolic equation}\label{WOSmethod}
In this part, we provide a detailed implementation of a stochastic method for computing the approximate solution $u(\bx_j, t_{n,\ell})$ at the given points $\{\bx_j\}_{j=0}^{N_x}$ and $\{t_{n,\ell}\}_{\ell=0}^{M_n}$. For simplicity, we only consider the case $\mathcal{L} = -\Delta$, while noting that the stochastic algorithm also applies to PDEs involving the following more general integro-differential operators:
\begin{equation*}
\begin{split}
\mathcal{L}[u](t,\bx)  =\,& \frac{1}{2}\mathrm{Trace}\!\big( \sigma(t,\bx)\,\sigma(t,\bx)^{\top}\,\mathrm{Hess}_{\bx}u(t,\bx) \big) 
+ \big\langle \mu(t,\bx), \nabla_{\bx}u(t,\bx) \big\rangle \\
&\quad + \dint_{\mathcal{D}}\!\Big[ u\!\big(t,\bx+c(t,\bx,\bz)\big) - u(t,\bx) \Big]\varphi(\bz)\,\mathrm{d}\bz,
\end{split}
\end{equation*}
where further details on the stochastic methods can be found in \cite{Sheng2025Efficient}.
 It is also worth noting that the algorithm presented in this section is designed for the high-dimensional case, whereas the one-dimensional case can be regarded as a  special case.

Benefiting from the fact that each iteration only requires solving a linear parabolic equation (refer to \eqref{iterscheme}), we only need to consider the following linear parabolic equation over $I=(0,T]$:
\begin{equation}\label{ufg_0}
\begin{cases}
\partial_t u(\bx,t)-\Delta u(\bx,t) = f(\bx,t),\;&(\bx,t) \in \ \Omega \times I ,\\[4pt]
u(\bx,t)=g(\bx,t),\quad & (\bx,t) \in \  \partial \Omega \times I ,\\[4pt]
u(\bx,0)=u_0(\bx),\quad &\bx \in \Omega,
\end{cases}
\end{equation}
Thanks to the Feynman-Kac formula (cf.\ \cite{Freidlin1992,Su2023AN}), there exists a unique continuous solution to \eqref{ufg_0}, which can be expressed as
\begin{equation}\label{fkx_ufg_0}
\begin{aligned}
u(\bx,t) &= \mathbb{E}_{X_{0}=\bm{x}}\Big[ u_0(X_t)\mathbb{I}_{\tau_{\Omega}^{\bx}>t}+g(X_{\tau_{\Omega}^{\bx}},\tau_{\Omega}^{\bx})\mathbb{I}_{\tau_{\Omega}^{\bx}\leq t}+ \int_{0}^{t \wedge \tau_{\Omega}^{\bx}}f(X_{s},s)\,{\rm d}s\Big],    
\end{aligned}
\end{equation}
where $\{X_t\}_{t\ge 0}$ denotes the Brownian motion with initial location $X_0=\bx$, and $\tau^{\bx}_{\Omega} = \inf \{ s| X_{s} \notin \Omega \}$ represents the first exit time from $\Omega$.

Motivated by the multistep strategy in the temporal direction \eqref{multischeme}, we successively compute $u^n$ for $n = 1,2,\dots, N_h$, and then formulate the following parabolic problem on the time interval $\T_n = (t_{n-1}, t_n]$:
\begin{equation}\label{uf_1}
\begin{cases}
\partial_t u^n(\bx,t) - \Delta u^n(\bx,t) = f(\bx,t), & (\bx,t) \in \Omega \times \T_n, \\[4pt]
u^n(\bx,t) = g(\bx,t), & (\bx,t) \in \partial\Omega \times \T_n, \\[4pt]
u^n(\bx,t_{n-1}) = u_\ast^{n-1}(\bx,t_{n-1}), & \bx \in \Omega,
\end{cases}
\end{equation}
where $u_\ast^{n-1}(\bx,t_{n-1})$ denotes the numerical solution obtained on the previous interval $(t_{n-2},t_{n-1}]$ and evaluated at $t_{n-1}$, serving as the initial data for $\T_n$.  
Using the Feynman-Kac formula \eqref{fkx_ufg_0}, the probabilistic representation of the solution to \eqref{uf_1} can be written as  
\begin{equation}\label{fkx_2}
\begin{split}
u(\bx,t) &= \mathbb{E}_{X_{t_{n-1}}=\bx} \Big[\, 
u_\ast^{n-1}\!\big(X_t,t_{n-1}\big)\,\mathbb{I}_{\tau_{\Omega}^{\bx}>t} 
+ g\big(X_{\tau_{\Omega}^{\bx}},\tau_{\Omega}^{\bx}\big)\,\mathbb{I}_{\tau_{\Omega}^{\bx}\leq t} + \int_{t_{n-1}}^{\,t \wedge \tau_{\Omega}^{\bx}} 
f\big(X_{s},s\big)\,\mathrm{d}s 
\Big],
\end{split}
\end{equation}
where the Brownian motion $\{X_t\}_{t\ge t_{n-1}}$ is initialized at $X_{t_{n-1}} = \bx$.

The probabilistic representation above indicates that, combined with the Monte Carlo method, the numerical solution can be directly obtained by simulating the stochastic process $\{X_t\}_{t \ge t_{n-1}}$. We next introduce the WoS algorithm, which simulates the stochastic process by tracing the motion of a sequence of balls within the domain $\Omega$.  
To compute the numerical solution at the Legendre-Gauss-Lobatto nodes $\{u(\bx_j,t_{n,\ell})\}_{\ell=1}^{M_n}$ within the subinterval $\T_n$, for each temporal node $t_{n,\ell}$ we need to simulate the sequence $\{X_{t_{n,\ell,q}}\}_{q=1}^{N_\ast}$. Therefore, we now present the following result describing the propagation from $X_{t_{n,\ell,q-1}}$ to $X_{t_{n,\ell,q}}$ (cf.\ \cite{Sheng2023}).

\begin{lemma}\label{Xlocation}
Let $X_{t_{n,\ell,q-1}}$ be the location of a $d$-dimensional Brownian motion at time $t_{n,\ell,q-1}$.
Let $\theta_1,\ldots,\theta_{d-1}$ be random variables such that the induced direction is uniformly
distributed on the unit sphere $\mathbb{S}^{d-1}$.
Then the location of the Brownian motion $X_{t_{n,\ell,q}}$ at time $t_{n,\ell,q}$ is given by the random variable
\begin{equation}
\label{Xi}
X_{t_{n,\ell,q}} = X_{t_{n,\ell,q-1}} + J\cdot
\begin{bmatrix}
\cos\theta_{1}\\
\sin\theta_{1}\cos\theta_{2}\\\
\cdots\cdots\\
\sin\theta_{1}\cdots\sin\theta_{d-2}\sin\theta_{d-1}
\end{bmatrix},
\end{equation}
where $J= \sqrt{2d\,\Delta_{n,\ell}}$ denotes the magnitude of the Brownian increment over the time step
$\Delta_{n,\ell}=t_{n,\ell,q}-t_{n,\ell,q-1}$.
\end{lemma}

For any given $\bm{x}_j \in \Omega$ and $t_{n,\ell}$, we draw a ball $\mathbb{B}_r^d(\bm{x}_j)$ inside the domain $\Omega$.  
Then, using Lemma\,\ref{Xlocation}, we can readily obtain the locations of a sequence of balls with radius $r$ centered at  
$$\big\{\bx_j = X_{t_{n,\ell,0}} ,X_{t_{n,\ell,1}} , \cdots ,X_{t_{n,\ell,N_\ast}}\big\}.$$
For brevity, we denote $X_{t_{n,\ell,q}}$ by $X_q$ and $t_{n,\ell,q}$ by $t_q$ in what follows.  
We now only need to determine whether the center of the $q$-th ball, $X_q$ with $1 \leq q \leq N_\ast$, lies within the domain $\Omega$.  
For this purpose, we define the following index set
\begin{equation} \label{indexL}
L=\begin{cases}
N_\ast,\;\; &\text{if}\; X_{q}\in \Omega,\;\forall 1\leq q \leq N_\ast,\\[2pt]
p,\;\; &\text{if}\; X_{q}\in \Omega,\;\forall 1\leq q \leq p<N_\ast\; \&\;X_{p+1}\notin \Omega.
\end{cases}  
\end{equation}
The index $L$ identifies the last point of the path inside $\Omega$.  
Using this, the path integral term in \eqref{fkx_2} is approximated by the quadrature rule as  
\begin{equation}\label{interIQ}
\begin{aligned}
\mathcal{Q}_{L}[f] &= \Delta_{n,\ell}\sum_{q=0}^{L-1}f(X_{q},t_q).   
\end{aligned}
\end{equation}
For a fixed spatial-temporal point $(\bx_j,t_{n,\ell})$, the Monte Carlo approximation of the solution of \eqref{fkx_2} is defined as  
\begin{equation*}
u^n_\ast(\bm{x}_j,t_{n,\ell}) = \frac{1}{M}\sum_{i=1}^M S_i,
\end{equation*}
where $S_i$ denotes the $i$-th realization of the stochastic solution, whose value is determined by the trajectory of the stochastic process as follows %.  
%If $X^i_{q}\in \Omega$ for all $1 \leq q \leq N_\ast$, then  
\begin{equation}\label{Phi_ast}
S_i=\begin{cases}
u_\ast^{n-1}\big(X_{N_\ast}^i, t_{N_\ast}\big) + \mathcal{Q}^i_{N_\ast}[f],\;\; &\text{if}\; X_{q}\in \Omega,\;\forall 1\leq q \leq N_\ast,\\[6pt]
g\big(X_{p}^i, t_{p}\big) + \mathcal{Q}^i_p[f],\;\; &\text{if}\; X_{q}\in \Omega,\;\forall 1\leq q \leq p<N_\ast\; \&\;X_{p+1}\notin \Omega.
\end{cases}  
\end{equation}

\begin{wrapfigure}{r}{0.45\textwidth}
\vspace{-10pt}
\centering
\includegraphics[width=0.41\textwidth]{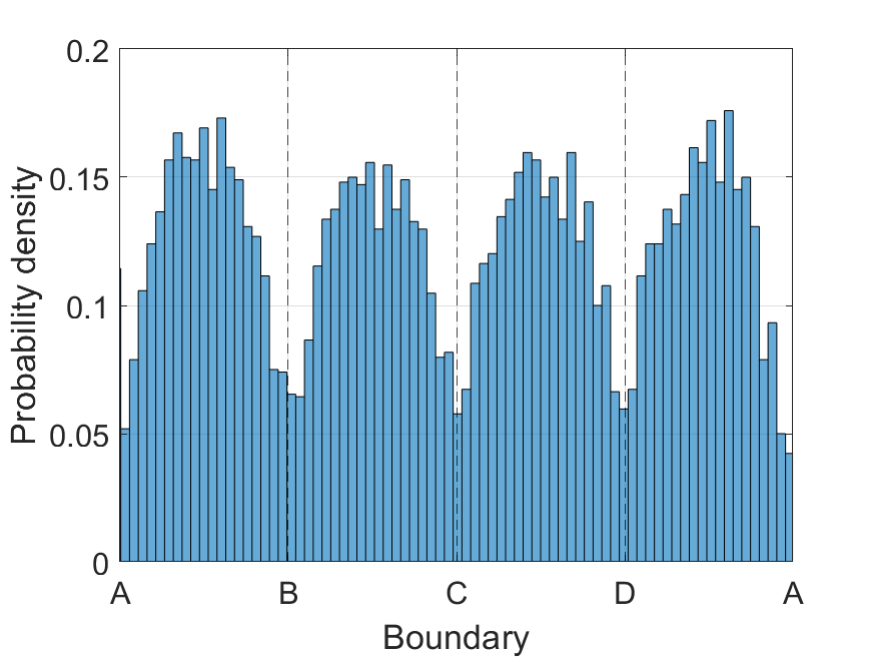}
\caption{Probability distribution of boundary hitting locations along $\partial\Omega$ for $\Omega=[-1,1]^2$.}
\label{fig:boundary_distribution}
\vspace{-10pt}
\end{wrapfigure}

We observe that the estimator~\eqref{Phi_ast} consists of two cases:
(i) trajectories that remain inside the domain $\Omega$ up to the maximal step $N_\ast$, as described in the first line of~\eqref{Phi_ast};
(ii) trajectories that exit the domain at some step $p+1\le N_\ast$, as described in the second line.
The relative frequency of these two cases determines the contribution of each branch in~\eqref{Phi_ast}.
To examine this behavior, we simulate $10^4$ WoS trajectories starting from randomly sampled points in $\Omega=[-1,1]^2$.
Table~\ref{tab:boundary_hitting_rate} reports the resulting empirical boundary exit probabilities for different final times $T$ and maximal step numbers $N_\ast$ (case (ii)). The results show that the boundary hitting probability increases with both $T$ and $N_\ast$.
In particular, for $T=1$, the exit probability exceeds $97\%$ for all tested values of $N_\ast$.

\begin{table}[htbp]
\centering
\caption{Boundary hitting probability of WoS trajectories for different $T$ and $N_\ast$.}
\label{tab:boundary_hitting_rate}
\begin{tabular}{c|ccc}
\hline
$T \backslash N_\ast$ & $20$ & $50$ & $100$ \\
\hline
$0.2$ & 64.88\% & 69.28\% & 71.07\% \\
$0.5$ & 87.44\% & 90.52\% & 92.46\% \\
$1.0$ & 97.25\%  & 98.76\%  & 98.97\% \\
\hline
\end{tabular}
\end{table}
% \begin{figure}[ht!]
% \centering
% \includegraphics[width=0.50\textwidth]{figures/boundary_distribution.pdf}
% \caption{Probability distribution of boundary hitting locations along $\partial\Omega$ for $\Omega=[-1,1]^2$.}
% \label{fig:boundary_distribution}
% \end{figure}

In addition to the boundary exit rate, we also examine the spatial distribution of boundary hitting locations.
To this end, for the square domain $\Omega=[-1,1]^2$, the vertices on $\partial\Omega$ are defined in counterclockwise order as
$A=(-1,-1)$, $B=(1,-1)$, $C=(1,1)$, and $D=(-1,1)$.
Fig.\;\ref{fig:boundary_distribution} depicts the empirical distribution of boundary hitting locations obtained from the same set of WoS simulations for the representative case $T=1$ and $N_\ast=100$.
The empirical distribution indicates an approximately symmetric density along $\partial\Omega$.
In particular, lower densities are observed in the vicinity of the corners, whereas higher densities are concentrated along the central portions of the edges. Moreover, the distributions along the four edges are broadly comparable.

\subsection{Space-time spectral interpolation}\label{STinterpolation}
A crucial step in the spectral Monte Carlo algorithm is accurately reconstructing the original function and its derivatives of various orders from coarse data at each iteration.  
To this end, we first discuss several function reconstruction techniques based on space-time spectral methods.  
For simplicity, we focus here on reconstruction using Legendre polynomials.  
It is clear that this approach can be readily extended to Chebyshev, Jacobi, or even radial basis functions, with the specific choice determined by the structure of the governing equation. Moreover, for ease of presentation, we restrict our discussion to the one-dimensional spatial case, namely $\Omega = (-1,1)$ with $g(\bx,t)=0$ in \eqref{ufg_0}.  
For nonhomogeneous boundary conditions, i.e., $g(\bx,t)\neq 0$, a simple homogenization leads to an equivalent homogeneous problem, to which the following method can be directly applied.

Let $L_p(x)$, $x \in (-1,1)$, denote the standard Legendre polynomial of degree $p$.  
In the temporal direction $t \in \T_n$, we use shifted Legendre polynomials as basis functions, defined by  
\begin{equation} 
L_{n,p}(t) = L_p\Big(\frac{2t - t_{n-1} - t_n}{h_n}\Big), \quad p= 0, 1, 2, \cdots 
\end{equation}
Let $\{t_{n,k}, \omega_{n,k}\}_{k=0}^{M_n}$ denote the set of shifted Legendre-Gauss-Lobatto (LGL) quadrature nodes and weights on the subinterval $\mathcal{T}_n$.
It is evident that the Lagrange interpolation basis in the time direction can be reformulated as
\begin{equation}\label{timeLag}
    h_k(t)=\prod_{\substack{0 \leq j \leq {M_n} \\ j \neq k}} \frac{t-t_{n,j}}{t_{n,k}-t_{n,j}}=\sum_{q=0}^{M_n}b_{qk}L_{n,q}(t),\quad \text{where}\,\,  b_{qk}=\frac{1}{\gamma_q}L_{n,q}(t_{n,k})\omega_{n,k},
\end{equation}
with $\gamma_q=\frac{2}{2q+1}$, $q=0,\cdots,M_n-1$ and $\gamma_{M_n}=\sum_{k=0}^{M_n}(L_{n,M_n}(t_{n,k}))^2\omega_{n,k}.$ 
For the spatial approximation under homogeneous boundary conditions, we employ the Babuška-Shen basis functions (cf. \cite{Shen2011}) defined by  $
\phi_p(x) = \frac{2(p+1)}{2p+3} \left(L_p(x) - L_{p+2}(x)\right),$  
where $L_p(x)$ denotes the Legendre polynomial of degree $p$. This basis can also be equivalently written as the generalized Jacobi function  
$$
\phi_p(x) = (1 - x^2) P_p^{(1,1)}(x),\;\;\;x\in(-1,1),
$$  
where $P_p^{(1,1)}(x)$ are the classical Jacobi polynomials.
Let $\{x_j,\omega_j\}_{j=0}^{N_x}$ denote the Jacobi-Gauss nodes and weights associated with the parameters $(1,1)$.  
One may readily verify that the Lagrange interpolation basis associated with $\{x_j\}$ in the spatial variable can be expressed in terms of generalized Jacobi functions as  
\begin{equation}\label{spacelag}
l_j(x)
= \frac{1 - x^2}{1 - x_j^2}
\prod_{\substack{0 \le m \le N_x \\ m \neq j}}
\frac{x - x_m}{x_j - x_m}
= \sum_{p=0}^{N_x} a_{pj}\,\phi_p(x), \quad
a_{pj} =\frac{1}{\tilde\gamma_p}(1 - x_j^2)^{-1} P_p^{(1,1)}(x_j)\,\omega_j,
\end{equation}
where $ \tilde\gamma_p=\frac{8(p+1)}{(2p+3)(p+2)}$. 
Given $\{u(x_j, t_{n,k})\}_{0 \le j \le N_x}^{0 \le k \le M_n}$,  
where $\{x_j\}_{j=0}^{N_x}$ denote the Jacobi-Gauss nodes and $\{t_{n,k}\}_{k=0}^{M_n}$ denote the shifted LGL nodes over $\mathcal{T}_n$,  
one may readily deduce from \eqref{timeLag} and \eqref{spacelag} that the reconstructions of $u_\ast$, $-\Delta u_\ast$, and $\partial_t u_\ast$ take the form (cf.~\cite[Prop.~2.2]{Yuan2025}):  
\begin{equation}\label{u_N}
\begin{split}
u_{\ast}(x,t) &= \sum_{p=0}^{N_{x}}\sum_{q=0}^{M_n}\tilde{u}_{pq}\,\phi_{p}(x)\,L_{n,q}(t), 
\quad -\Delta u_\ast(x,t) = \sum_{p=0}^{N_{x}}\sum_{q=0}^{M_n}\tilde{u}_{pq}^{x}\,P_{p}^{(1,1)}(x)\,L_{n,q}(t),\\[4pt]
\partial_t u_\ast(x,t) &= \sum_{p=0}^{N_{x}}\sum_{q=0}^{M_n}\tilde{u}^{t}_{pq}\,P_{p}^{(1,1)}(x)\,L_{n,q}(t),
\end{split}
\end{equation}
where the modal coefficients are given by
\begin{eqnarray*}
&&\tilde{u}_{pq} = \sum_{j=0}^{N_x}\sum_{k=0}^{M_n} u(x_j,t_{n,k})\,a_{pj}\,b_{qk}, 
\quad \tilde{u}_{pq}^{x} = \sum_{j=0}^{N_{x}}\sum_{k=0}^{M_n} (p+1)\,u(x_j,t_{n,k})\,a_{pj}\,b_{qk},\\[4pt]
&&\tilde{u}_{pq}^{t} = \sum_{\substack{l=q \\ l+q\;\text{odd}}}^{M_n}\sum_{j=0}^{N_{x}}\sum_{k=0}^{M_n} u(x_j,t_{n,k})\,a_{pj}\,b_{lk}\,\frac{2(2q+1)}{h_n},
\end{eqnarray*}
and $\{a_{pj}\}$ and $\{b_{qk}\}$ are defined in \eqref{timeLag} and \eqref{spacelag}, respectively.
\begin{rem}
%For the spatial discretization, the one-dimensional basis functions $\phi_p(x)$ naturally extend to higher dimensions through the tensor-product construction $\phi_{p_1}(x_1)\cdots\phi_{p_d}(x_d)$.  
%This straightforward generalization allows the same methodology to be applied in multi-dimensional settings; however, for brevity, we omit these details here.
For the spatial discretization, the one-dimensional basis functions $\phi_p(x)$ naturally extend to higher dimensions through the tensor-product construction $\phi_{p_1}(x_1)\cdots\phi_{p_d}(x_d)$. 
This straightforward generalization allows the same methodology to be applied in multidimensional settings. For brevity, we omit the details.
\end{rem}

\subsection{Implementation of time-stepping spectral Monte Carlo algorithm}\label{algormxt}
 In this subsection, we focus on the nonlinear equation~\eqref{mainproblem}, in which the operator $\mathcal{L}$ is taken to be the classical Laplace operator. We demonstrate that the proposed stochastic algorithm attains high accuracy through iterative refinement, while remaining robust and computationally efficient in long-time simulations. In what follows, we aim to compute the numerical solution $u^{n,(k)}_\ast(\bx,t)$, where the superscript $n$ denotes the temporal subinterval $\mathcal{T}_n$, and $k$ specifies the $k$-th iteration.

We first introduce the Picard iteration scheme equation \eqref{multischeme} nonlinear.:
\begin{equation}\label{linearized_func}
\begin{cases}
\partial_t u^{n,(k)}(\bx,t) - \Delta u^{n,(k)}(\bx,t) = f\big(u^{n,(k-1)}_\ast(\bx,t)\big), & (\bx,t) \in \Omega \times \T_n, \\[4pt]
u^{n,(k)}(\bx,t) = 0, & (\bx,t) \in \partial \Omega \times \T_n, \\[4pt]
u^{n,(k)}(\bx,t_{n-1}) = u_\ast^{n-1}(\bx), & \bx \in \Omega,
\end{cases}
\end{equation}
where $u_\ast^{n-1}(\bx)$ denotes the refined initial solution for the current time interval $\T_n$, obtained from the numerical approximation on the preceding interval $\T_{n-1}$.  
In practice, however, this straightforward iterative scheme is not directly applicable.  
To address this issue, we introduce the residual function $\varepsilon^{n,(k)} = u^{n,(k)} - u^{n,(k-1)}_\ast$ and substitute it into the scheme \eqref{linearized_func}, yielding an equivalent residual formulation:
\begin{equation}\label{ufg_ep1}
\begin{cases}
\partial_t \varepsilon^{n,(k)}(\bx,t)-\Delta \varepsilon^{n,(k)}(\bx,t) =f^{(k)}(u^{n,(k-1)}_\ast(\bx,t)),\;&(\bx,t) \in \ \Omega \times \T_n ,\\[4pt]
\varepsilon^{n,(k)}(\bx,t)=0,\quad & (\bx,t) \in \  \partial \Omega \times \T_n ,\\[4pt]
\varepsilon^{n,(k)}(\bx,t_{n-1})=0,\quad &\bx \in \Omega,
\end{cases}
\end{equation}
where $f^{(k)}(u^{n,(k-1)}_\ast(\bx,t)) = f(u^{n,(k-1)}_\ast(\bx,t))-\partial_t u^{n,(k-1)}_\ast(\bx,t)+\Delta u^{n,(k-1)}_\ast(\bx,t).$ 

Notably, the right-hand side of \eqref{ufg_ep1} contains terms such as $u^{n,(k-1)}_\ast(\bx,t)$, \\$\partial_t u^{n,(k-1)}_\ast(\bx,t)$, and $\Delta u^{n,(k-1)}_\ast(\bx,t)$, whose evaluation inherently depends on accessing solution information at random locations arising in the stochastic algorithm.  
These quantities are efficiently recovered using the reconstruction techniques outlined in Section~\ref{STinterpolation}.
It is evident that the above equation is linear, which allows the walk-on-spheres method described in Section \ref{WOSmethod} to be seamlessly applied.  
In this way, the numerical solution of \eqref{ufg_ep1} can be efficiently evaluated at Gauss-type quadrature points, yielding the set of values $\{\varepsilon^{n,(k)}_\ast(\bx_j,t_{n,\ell})\}$ as below
\begin{equation}\label{varepsilon_1}
\varepsilon^{n,(k)}_\ast(\bx_j,t_{n,\ell})=\frac{1}{M}\sum_{i=1}^M \mathcal{Q}_L^i\big[f^{(k)}(u^{n,(k-1)}_\ast(\bx,t))\big],
\end{equation}
where $\mathcal{Q}_L[\cdot]$ is defined in \eqref{interIQ}.  Then the new approximation is given by 
$u^{n,(k)}_\ast(\bx,t) $ $= u^{n,(k-1)}_\ast(\bx,t) + \varepsilon^{n,(k)}_\ast(\bx,t)$.  
Finally, this correction procedure is iterated until
$\max_{j,\ell}\{\varepsilon^{n,(k)}_\ast(\bx_j,t_{n,\ell})\}$ is below the prescribed tolerance.

For the reader’s convenience, we briefly summarize the above derivations, which lead to the spectral Monte Carlo algorithm presented in Algorithm~\ref{algo_2}.
\begin{algorithm}
\caption{ $hp$-version time-stepping spectral Monte Carlo algorithm for \eqref{ufg_0}.}
\label{algo_2}
\begin{algorithmic}
\REQUIRE  $M,\epsilon,N_x,M_n,k_0,\T_h$;
%\STATE Initialize  $n = 1$;
\FOR{$ n=1,2,\cdots N_h$}  
\WHILE{$\varepsilon^{(k)}>\epsilon$ {\bf or} $k\leq k_0$}

\smallskip

\STATE Step 1. Construct $u_\ast^{n,(k-1)}(\bx,t)$, $\partial_tu_\ast^{n,(k-1)}(\bx,t)$, $\Delta u_\ast^{n,(k-1)}(\bx,t)$ by \eqref{u_N};

\FOR{$j=1:N_x$  {\bf(in parallel)}}
\FOR{$\ell=1:M_n$   {\bf(in parallel)}}
\STATE Step 2. Compute the residual function $\varepsilon^{n,(k)}_\ast(\bx_j,t_{n,\ell})$ by \eqref{varepsilon_1};

\smallskip
\STATE  Step 3. Update $u^{n,(k)}_\ast(\bx_j,t_{n,\ell})= u^{n,(k-1)}_\ast(\bx_j,t_{n,\ell})+\varepsilon^{n,(k)}_\ast(\bx_j,t_{n,\ell})$; 
\ENDFOR
\ENDFOR

\STATE  Update $k =k+1$

\ENDWHILE
\STATE   Step 4. Set $u^{n,(k)}_\ast(\bx,t_{n-1}) = u_\ast^{n-1}(\bx,t_{n-1})$;
\ENDFOR
\ENSURE  $\{u(\bx_j,t_{n,\ell})\approx u^{n,(k)}_\ast(\bx_j,t_{n,\ell})\}_{j,\ell}$ with $n=1,\cdots,N_h$.
\end{algorithmic}
\end{algorithm}
\begin{rem}{
In practice, initializing with a simple constant, e.g., $u_\ast^{(0)} = c$ for a random $c$, often suffices, and the algorithm still converges, which demonstrates its robustness to the choice of initial data.}
\end{rem}

\begin{rem}{ 
It is clear that the computational efficiency of the iterative algorithm is inherently governed by the convergence rate of the iteration scheme.  
One viable strategy for enhancing its performance is to couple the proposed iteration with classical acceleration techniques.  
For instance, incorporating Aitken extrapolation into the Picard iteration via a weighted averaging procedure yields approximately a 15\% improvement in convergence while preserving numerical stability.  
Although this balanced design provides a noticeable enhancement in iteration efficiency for nonlinear problems, achieving an order-of-magnitude reduction in the required iteration count remains challenging at this stage.}
\end{rem}

\begin{rem} 
We summarize the computational complexity of the proposed \\ stochastic algorithm.
To this end, we first consider the per-iteration cost at a single time step, which consists of the following two components:
\begin{itemize}
  \item By exploiting the intrinsic parallelism of stochastic algorithms, the computation of the values at the spatial/temporal grid points using the WoS method requires $\mathcal{O}\!\left(N_x M_n M / P_{\rm core}\right)$ operations.

  \item The \emph{reconstruction process} involves transformations between nodal values and expansion coefficients,
including both the nodal-to-modal and modal-to-nodal mappings, with a total computational cost of
$\mathcal{O}\!\left(N_x^{2} M_n + N_x M_n^{2}\right)$. 
\end{itemize}
Here, $N_x$ and $M_n$ denote the spatial and temporal degrees of freedom in each time slab, respectively, 
$M$ is the number of Monte Carlo samples, and $P_{\rm core}$ denotes the number of CPU cores.
Taking into account $N_h$ time steps and $k_0$ outer iterations, the overall computational
complexity of the algorithm is
$$
\mathcal{O}\!\left(
N_{h} k_{0}N_{x}^{2} M_{n}
+ N_{h} k_{0}N_{x} M_{n}^{2}
+ N_{h} k_{0}N_{x} M_{n} M/P_{\rm core}
\right).
$$
Remarkably, the above complexity estimates can be further improved by simply replacing the Legendre polynomial basis with Chebyshev polynomials. 
In this case, fast Fourier transforms can be exploited to achieve quasi-linear complexity, leading to
$$
\mathcal{O}\!\left(
N_h k_0 M_n \,N_x \log N_x
+ N_h k_0 N_x M_n \log M_n
+ N_h k_0 N_x M_nM/P_{\rm core}
\right).
$$
%For comparison, classical space--time spectral methods for semilinear parabolic equations {\rm\cite{shen2007}} result in a fully coupled nonlinear algebraic system with $N_x M_n$ degrees of freedom on each time slab.
%Such systems are typically solved by Jacobian-free Newton--Krylov methods, where each nonlinear iteration involves the evaluation of the residual and Jacobian--vector products with cost
%$\mathcal{O}\!\left( N_x^2 M_n^2 \right)$.
%Accounting for $N_h$ time slabs and $k_0$ outer nonlinear iterations, the overall computational complexity is
%$\mathcal{O}\!\left( N_h k_0 N_x^2 M_n^2 \right)$.
\end{rem}

\section{Error estimates}\label{sect3}
\setcounter{lem}{0} \setcounter{thm}{0}  \setcounter{rem}{0}
In this section, we analyze and characterize the convergence of the $hp$-version spectral Monte Carlo iterative algorithm for semi-linear parab- olic equations, with the goal of deriving error bounds in the sense that
\begin{equation}\label{Enkinfty}
E^\infty_{n,k}:=\max_{0\leq j\leq N_x,0\leq \ell \leq M_n} \big|\mathbb{E}\big(u(x_j,t_{n,\ell})-u_\ast^{n,(k)}(x_j,t_{n,\ell})\big)\big|.%\big\|_{L^\infty}.
\end{equation}
where $u^{n,(k)}_\ast(x_j,t_{n,\ell})$ denotes the numerical solution at $k$-th iteration over the sub-interval $\T_n$ at the Legendre-Gauss points $\{x_j\}_{j=0}^{N_x}$ and shifted Legendre-Gauss-Lobatto points $\{t_{n,\ell}\}_{\ell=0}^{M_n}$. Throughout this section, we assume that the nonlinear function $f$ satisfies the following Lipschitz condition:
\begin{equation}\label{Lipschitz con}
\big|f(u_1)-f(u_2)\big| \leq C_L |u_1-u_2|,
\end{equation}
where $C_L \geq 0$ is the Lipschitz constant.

We first present the stability properties of the space-time interpolation operator $\mathcal{I}_{N_x}\mathcal{I}_{M_n}^t$, where $\mathcal{I}_{N_x}u(x)=\sum_{j=0}^{N_x}u(x_j)l_j(x)$ and $\mathcal{I}_{M_n}^tu(t)=\sum_{k=0}^{M_n}u(t_{n,k})h_k(t)$, with $l_j(x)$ and $h_k(t)$ defined in \eqref{timeLag} and \eqref{spacelag}, respectively.
\begin{lem}\label{lem3.2} For any $u\in L^\infty(\Omega\times\T_n)$, there holds
\begin{equation}\label{Instability}
\big\| \mathcal{I}_{N_x}\mathcal{I}_{M_n}^t  u \big\|_{L^{\infty}}  \leq c N_x^{\frac{1}{2}} M_n^{\frac{1}{2}} \|u \|_{L^{\infty}},
\end{equation}    
where $c$ is a positive constant independent of $M_n$, $N_x$, and $u$.
\end{lem}
\begin{proof}
According to \cite{Xiang2016}, the Lagrange interpolation bases $\{h_k(t)\}$ and $\{l_j(x)\}$ associated with the temporal nodes $\{t_{n,k}\}$ and spatial nodes $\{x_j\}$ admit the following bounds:
\begin{equation}\label{hjbound}
\max_{0 \leq k \leq M_n} |h_k(t)| = \mathcal{O}(M_n^{\frac12}), \quad 
\max_{0 \leq j \leq N_x} |l_j(x)| = \mathcal{O}(N_x^{\frac12}).
\end{equation}
It then follows that
\begin{equation*} 
\begin{aligned}
\big\| \mathcal{I}_{N_x}\mathcal{I}_{M_n}^t  u \big\|_{L^{\infty}} = \big\| \mathcal{I}_{M_n}^t\big(\mathcal{I}_{N_x}  u \big)\big\|_{L^{\infty}} \leq c N_x^{\frac{1}{2}}\big\| \mathcal{I}_{M_n}^t  u \big\|_{L^{\infty}} \leq c N_x^{\frac{1}{2}}(M_n+1)^{\frac{1}{2}} \|u\|_{L^{\infty}}.
\end{aligned}
\end{equation*}
This ends the proof.
\end{proof}

\begin{lem}\label{lem3.1}
Suppose that $\partial_x^{r-1}u(x,t)$ is absolutely continuous on $\Omega \times \T_n$ for some $r \geq 1$,  $\partial_x^r u$ has bounded variation ${\rm BV}[\partial_x^r u] < \infty$, and $u \in H^m(\T_n; L^2(\Omega))$ with $m \geq 1$. Then, there holds
\begin{equation}\label{interpolationerr_NxMn}
\big\| \mathcal{I}_{N_x}\mathcal{I}_{M_n}^t  u -  u \big\|_{L^{\infty}} \leq cN_x^{\frac12}h_n^{m-\frac12} M_n^{\frac12-m}\big\|\partial_t^mu\big\|_{L^2(\T_n;L^2(\Omega))}+ c N_x^{\frac12-r},
%cN_x^{1/2}h_n^{m-1/2} M_n^{1-m}\big\|\partial_t^mu\big\|_{\T_n}+ c N_x^{1/2-r}
\end{equation}
where $c$ is a positive constant independent of $r$, $m$, $h_n$, $M_n$, and $N_x$.
\end{lem}
\begin{proof}According to \cite[Theorem 3.2]{Wang2017}, for any $v\in H^m(\T_n)$ with integer $1\leq m\leq M_n+2$ and $M_n\geq0$, we have
\begin{equation*}
\big\| \mathcal{I}_{M_n}^t  v-v \big\|_{L^2(\T_n)}\leq ch_nM_n^{-m}\big\|\partial_t^mv\big\|_{L_{\chi_n^{m-1}}^2(\T_n)}\leq ch_n^m M_n^{-m} \big\|\partial_t^mv\big\|_{L^2{(\T_n)}},     
\end{equation*}
and
\begin{equation*}
\big\| \partial_t(\mathcal{I}_{M_n}^t  v-v) \big\|_{L^2(\T_n)}\leq cM_n^{1-m}\big\|\partial_t^mv\big\|_{L_{\chi_n^{m-1}}^2(\T_n)}\leq ch_n^{m-1} M_n^{1-m} \big\|\partial_t^mv\big\|_{L^2{(\T_n)}},     
\end{equation*}
where $H^m(\mathcal{T}_n)$ denotes the standard Sobolev space and the weight function is defined by $\chi_n^m(t) = (t_n - t)^m (t - t_{n-1})^m$.
Thus, by applying the Sobolev-type inequalities together with the above two estimates, we obtain
\begin{equation}\label{IMnu_u}
\begin{aligned}
\big\|\mathcal{I}_{M_n}^tv-v \big\|_{L^\infty(\T_n)}&\leq c\big\|\mathcal{I}_{M_n}^tv-v\big\|^{1/2}_{L^2(\T_n)}\big\|\mathcal{I}_{M_n}^tv-v\big\|^{1/2}_{H^1(\T_n)}\\[4pt]
&\leq ch_n^{m-1/2} M_n^{1/2-m} \big\|\partial_t^mv\big\|_{L^2{(\T_n)}}.    
\end{aligned}
\end{equation}
Recall that, as established in \cite{Xiang2016}, if $\partial_x^{r-1}u(x,t)$ is absolutely continuous on $\Lambda \times \T_n$ for some $r \geq 1$, and the total variation satisfies ${\rm BV}[\partial_x^{r}u] < \infty$, then 
\begin{equation}\label{INxu_u}
\big\| \mathcal{I}_{N_x} u - u \big\|_{L^\infty} \leq c N_x^{1/2 - r}.
\end{equation}
Then, we obtain from \eqref{IMnu_u}, \eqref{INxu_u}, and \eqref{hjbound} that
\begin{equation*}
\begin{aligned}
\big\| \mathcal{I}_{N_x} \mathcal{I}_{M_n}^t  u -  u \big\|_{L^{\infty}}
&\leq \big\| \mathcal{I}_{N_x}(\mathcal{I}_{M_n}^t u - u) \big\|_{L^{\infty}} + \big\| \mathcal{I}_{N_x} u - u \big\|_{L^{\infty}}\\[4pt]
&\leq  N_x^{\frac{1}{2}}\big\| \mathcal{I}_{M_n}^t  u - u \big\|_{L^{\infty}}+ \big\| \mathcal{I}_{N_x}  u - u \big\|_{L^{\infty}} \\[4pt]
&\leq cN_x^{\frac12}h_n^{m-\frac12} M_n^{\frac12-m}\big\|\partial_t^mu\big\|_{L^2(\T_n;L^2(\Omega))}+ c N_x^{1/2-r}.
\end{aligned}
\end{equation*}
This ends the proof.
\end{proof}

To avoid ambiguity, we further rewrite $E^\infty_{n,k}$ in \eqref{Enkinfty} as $\|\mathbb{E}(u - u_\ast^{n,(k)})\|_{\ell^\infty}$, namely  
\begin{equation*}
\|\mathbb{E}(u - u_\ast^{n,(k)})\|_{\ell^\infty}:= \max_{0 \le j \le N_x,\; 0 \le \ell \le M_n} 
\Big| \mathbb{E} \big( u(x_j, t_{n,\ell}) - u_\ast^{n,(k)}(x_j, t_{n,\ell}) \big) \Big|.
\end{equation*}
Then, we can decompose the discrete error into the following four components.
\begin{lem}
Let $u(x,t)$ denote the exact solution of \eqref{ufg_0}, and let $u^{n,(k)}_\ast(x,t)$ be the numerical solution computed by {\rm Algorithm~\ref{algo_2}}. Then, it holds that
\begin{equation}
E^\infty_{n,k}=\big\|\mathbb{E}(u-u_\ast^{n,(k)})\big\|_{\ell^\infty} \leq B_1+B_2+B_3+B_4,
\end{equation}
where
% \begin{equation*}
% \begin{split}
\begin{eqnarray*}
&&B_1=\Big\|\mathbb{E}\Big[\mathbb{E}_{X_{t_{n-1}}=\bx}\Big[\int_{t_{n-1}}^{t_{n,\ell} \wedge \tau_{\Omega}^{\bx}}f(u) \,{\rm d}s-\mathcal{Q}_L\big[f(u)\big]\Big]\Big]\Big\|_{\ell^\infty},\\[4pt]
&&B_2=\Big\|\mathbb{E}\Big[\mathbb{E}_{X_{t_{n-1}}=\bx}\Big[\mathcal{Q}_L\big[f(u)-f(I_{N_x}I_{M_n}^t u)\big]\Big]\Big]\Big\|_{\ell^\infty}
,\\[4pt]
&&B_3=\Big\|\mathbb{E}\Big[\mathbb{E}_{X_{t_{n-1}}=\bx}\Big[\mathcal{Q}_L\big[f(I_{N_x}I_{M_n}^t u)-f(u_\ast^{n,(k-1)})\big]\Big]\Big]\Big\|_{\ell^\infty},\\[4pt]
&&B_4=\Big\|\mathbb{E}\Big[\mathbb{E}_{X_{t_{n-1}}=\bm{x}}\Big[ u^{n-1}(X_t,t_{n-1})\mathbb{I}_{\tau_{\Omega}^{\bx}>t}-u_\ast^{(n-1)}(X_t,t_{n-1})\mathbb{I}_{\tau_{\Omega}^{\bx}>t}\Big]\Big]\Big\|_{\ell^\infty}.
\end{eqnarray*}
% \end{split}
% \end{equation*} 

\end{lem}
\begin{proof}
For the sake of analysis, we rewrite the residual iteration scheme \eqref{ufg_ep1} into the equivalent Picard iteration form \eqref{linearized_func}. Accordingly, the numerical solution associated with \eqref{linearized_func} can be expressed as
\begin{equation}\label{uast}
u_\ast^{n,(k)}(\bx,t)=\frac{1}{M}\sum_{i=1}^{M}\Big[u_\ast^{n-1}(X_t^i,t_{n-1})\mathbb{I}_{\tau_{\Omega}^{\bx}>t}+g(X_{\tau_{\Omega}^{\bx}}^i,\tau_{\Omega}^{\bx})\mathbb{I}_{\tau_{\Omega}^{\bx}\leq t} + \mathcal{Q}^i_L\big[f(u_\ast^{n,(k-1)})\big]\Big]. 
\end{equation}
Due to the mutual independence of the stochastic processes, we have
\begin{equation}
\mathbb{E}\Big[\frac{1}{M}\sum_{i=1}^{M} \big[g(X_{\tau_{\Omega}^{\bx}}^i,\tau_{\Omega}^{\bx})\big]\Big] 
= \frac{1}{M}\sum_{i=1}^{M}\mathbb{E}\big[g(X_{\tau_{\Omega}^{\bx}}^i,\tau_{\Omega}^{\bx})\big] 
= \mathbb{E}\big[g(X_{\tau_{\Omega}^{\bx}},\tau_{\Omega}^{\bx})\big].
\end{equation}
Similarly, we have
\begin{equation}
\begin{split}
&\mathbb{E}\Big[\frac{1}{M}\sum_{i=1}^{M} \big[u_\ast^{(n-1)}(X_{t}^i,t_{n-1})\big]\Big] 
= \mathbb{E}\big[u_\ast^{(n-1)}(X_t,t_{n-1})\big],\\[4pt]
&\mathbb{E}\Big[\frac{1}{M}\sum_{i=1}^{M}\big[ \mathcal{Q}^i_L\big[f(u_\ast^{(n,k-1)})\big]\big]\Big]=\mathbb{E}\Big[ \mathcal{Q}_L\big[f(u_\ast^{(n,k-1)})\big]\Big].    
\end{split}
\end{equation}
Subtracting the numerical formulation \eqref{uast} from the exact expression \eqref{fkx_2}, we arrive at
% \begin{eqnarray*}
% &&E_{n,k}^{\infty} =\big\|\mathbb{E}\big(u-u_\ast^{n,(k)}\big)\big\|_{\ell^\infty} \\[4pt]
% &&\hspace{23pt}= \Big\|\mathbb{E}\Big[\mathbb{E}_{X_{t_{n-1}}=\bx}\big[u^{n-1}(X_{t_{n,\ell}},t_{n-1})\mathbb{I}_{\tau_{\Omega}^{\bx}>t_{n,\ell}}\\
% &&\hspace{23pt}\quad +g(X_{\tau_{\Omega}^{\bx}},\tau_{\Omega}^{\bx})\mathbb{I}_{\tau_{\Omega}^{\bx}\leq t_{n,\ell}}+\int_{t_{n-1}}^{t_{n,\ell} \wedge \tau_{\Omega}^{\bx}}f(u) \,{\rm d}s\big]\Big]\\[4pt]
% &&\hspace{23pt}\quad -\mathbb{E}\Big[\frac{1}{M}\sum_{i=1}^{M}\big[u_\ast^{n-1}(X_{t_{n,\ell}}^i,t_{n-1})\mathbb{I}_{\tau_{\Omega}^{\bx}>{t_{n,\ell}}}\\
% &&\hspace{23pt}\quad +g(X_{\tau_{\Omega}^{\bx}}^i,\tau_{\Omega}^{\bx})\mathbb{I}_{\tau_{\Omega}^{\bx}\leq {t_{n,\ell}}} + \mathcal{Q}^i_L\big[f(u_\ast^{n,(k-1)})\big]\big]\Big]\Big\|_{\ell^\infty}\\[4pt]
% &&\hspace{23pt}\leq \Big\|\mathbb{E}\Big[\mathbb{E}_{X_{t_{n-1}}=\bm{x}}\big[ u^{n-1}(X_{t_{n,\ell}},t_{n-1})\mathbb{I}_{\tau_{\Omega}^{\bx}>{t_{n,\ell}}}\\
% &&\hspace{23pt}\quad -u_\ast^{n-1}(X_{t_{n,\ell}},t_{n-1})\mathbb{I}_{\tau_{\Omega}^{\bx}>{t_{n,\ell}}}\big]\Big]\Big\|_{L^\infty}\\[4pt]
% &&\hspace{23pt}\quad +\Big\|\mathbb{E}\Big[\mathbb{E}_{X_{n-1}=\bx}\Big[\int_{t_{n-1}}^{{t_{n,\ell}}\wedge \tau_{\Omega}^{\bx}}f(u) \,{\rm d}s-\mathcal{Q}_L\big[f(u_\ast^{n,(k-1)})\big]\Big]\Big]\Big\|_{L^\infty}.
% \end{eqnarray*}
\begin{eqnarray*}
&&\hspace{-18pt}E_{n,k}^{\infty} =\big\|\mathbb{E}\big(u-u_\ast^{n,(k)}\big)\big\|_{\ell^\infty} \\[4pt]
&&\hspace{-18pt}= \Big\|\mathbb{E}\Big[\mathbb{E}_{X_{t_{n-1}}=\bx}\big[u^{n-1}(X_{t_{n,\ell}},t_{n-1})\mathbb{I}_{\tau_{\Omega}^{\bx}>t_{n,\ell}}+g(X_{\tau_{\Omega}^{\bx}},\tau_{\Omega}^{\bx})\mathbb{I}_{\tau_{\Omega}^{\bx}\leq t_{n,\ell}}+\int_{t_{n-1}}^{t_{n,\ell} \wedge \tau_{\Omega}^{\bx}}f(u) \,{\rm d}s\big]\Big]\\[4pt]
&&\hspace{-18pt} \quad -\mathbb{E}\Big[\frac{1}{M}\sum_{i=1}^{M}\big[u_\ast^{n-1}(X_{t_{n,\ell}}^i,t_{n-1})\mathbb{I}_{\tau_{\Omega}^{\bx}>{t_{n,\ell}}}+g(X_{\tau_{\Omega}^{\bx}}^i,\tau_{\Omega}^{\bx})\mathbb{I}_{\tau_{\Omega}^{\bx}\leq {t_{n,\ell}}} + \mathcal{Q}^i_L\big[f(u_\ast^{n,(k-1)})\big]\big]\Big]\Big\|_{\ell^\infty}\\[4pt]
&&\hspace{-18pt} \leq \Big\|\mathbb{E}\Big[\mathbb{E}_{X_{t_{n-1}}=\bm{x}}\big[ u^{n-1}(X_{t_{n,\ell}},t_{n-1})\mathbb{I}_{\tau_{\Omega}^{\bx}>{t_{n,\ell}}}-u_\ast^{n-1}(X_{t_{n,\ell}},t_{n-1})\mathbb{I}_{\tau_{\Omega}^{\bx}>{t_{n,\ell}}}\big]\Big]\Big\|_{L^\infty}\\[4pt]
&&\hspace{-18pt} \quad +\Big\|\mathbb{E}\Big[\mathbb{E}_{X_{n-1}=\bx}\Big[\int_{t_{n-1}}^{{t_{n,\ell}}\wedge \tau_{\Omega}^{\bx}}f(u) \,{\rm d}s-\mathcal{Q}_L\big[f(u_\ast^{n,(k-1)})\big]\Big]\Big]\Big\|_{L^\infty}.
\end{eqnarray*}
Next, we add the following auxiliary terms in the above equation
\begin{equation*}
\mathbb{E}\Big[\mathbb{E}_{X_{t_{n-1}}=\bx}\Big[\mathcal{Q}_L\big[f(u)\big]\Big]\Big],\;\; \mathbb{E}\Big[\mathbb{E}_{X_{t_{n-1}}=\bx}\Big[\mathcal{Q}_L\big[f(I_{N_x}I_{M_n}^t u)\big]\Big]\Big],
\end{equation*}
which allows us to derive the desired estimate. This completes the proof.
\end{proof}

\begin{lem}\label{lemma4.2}
Let $u\in L^\infty(\T_n;C^{0,1}(\Omega))$ and $u_t\in L^\infty(\T_n;L^{\infty}(\Omega))$, there holds
\begin{equation}\label{result_B_1}
B_1  \leq ch_n\Delta_{n,\ell}\|u_t\|_{L^\infty(\T_n;L^{\infty}(\Omega))}+ch_n \sqrt{\Delta_{n,\ell}}\|u\|_{L^\infty(\T_n;C^{0,1}(\Omega))},
\end{equation}
where $c$ is a constant independent of $h_n$, $\Delta_{n,\ell}$, and $u$.
\end{lem}
\begin{proof}
Recall from Section~\ref{WOSmethod} that, in the walk-on-spheres method, for any given $\bm{x}_j \in \Omega$ and $t_{n,\ell}$, we generate a sequence of spheres 
of radius $r$ centered at the points $\{\bx_j = X_{t_{n,\ell,0}}, X_{t_{n,\ell,1}}, \cdots,$ $ X_{t_{n,\ell,N_\ast}}\}.$ In what follows, we abbreviate $X_q := X_{t_{n,\ell,q}}$ and $t_q:=t_{n,\ell,q}$. 
Then, we obtain from \eqref{indexL} and \eqref{Lipschitz con} that
 \begin{eqnarray*}
&&\hspace{-20pt}B_1 = \Big\| \mathbb{E}\Big[\mathbb{E}_{X_{t_{n-1}}=\bx_j}\Big[\int_{t_{n-1}}^{t_{n,\ell} \wedge \tau_{\Omega}^{\bx}}f(u(X_s,s)) \,{\rm d}s-\mathcal{Q}_L\big[f(u)\big]\Big]\Big]\Big\|_{\ell^\infty}\\[4pt]
&&\hspace{-20pt}= \Big\| \mathbb{E}\Big[\sum_{q=0}^{L-1}\mathbb{E}_{X_q=\bx_{q}}\Big[\int_{t_{q}}^{t_{q+1}}\hspace{-2pt}f(u(X_s,s)) \,{\rm d}s\Big]\hspace{-1pt} -\hspace{-1pt}\mathbb{E}\Big[\sum_{q=0}^{L-1}\mathbb{E}_{X_{t_{q}}=\bx_{q}}\Big[\int_{t_{q}}^{t_{q+1}}\hspace{-2pt}f(u(X_s,t_{q})) \,{\rm d}s\Big]\Big]\Big\|_{\ell^\infty}\\[4pt]
&&\hspace{-20pt}\quad+ \Big\| \mathbb{E}\Big[\sum_{q=0}^{L-1}\mathbb{E}_{X_q=\bx_{q}}\Big[\int_{t_{q}}^{t_{q+1}}f(u(X_s,t_{q})) \,{\rm d}s\Big] -\sum_{q=0}^{L-1}\mathbb{E}_{X_{t_{q}}=\bx_{q}}\Big[ \Delta_{n,\ell}  f(u(X_{q},t_q))\Big]\Big]\Big\|_{\ell^\infty}\\[4pt]
&&\hspace{-20pt} \leq B_{1,1}+B_{1,2},
\end{eqnarray*}
where 
\begin{small}
\begin{equation*}
\begin{aligned}
&B_{1,1}=\Big\| \mathbb{E}\Big[\sum_{q=0}^{L-1}\mathbb{E}_{X_q=\bx_{q}}\Big[\int_{t_{q}}^{t_{q+1}}f(u(X_s,s))-f(u(X_s,t_{q})) \,{\rm d}s\Big]\Big]\Big\|_{\ell^\infty},\\[4pt]
&B_{1,2}=\Big\| \mathbb{E}\Big[\sum_{q=0}^{L-1}\mathbb{E}_{X_q=\bx_{q}}\Big[\int_{t_{q}}^{t_{q+1}}f(u(X_s,t_{q})) \,{\rm d}s\Big] -\sum_{q=0}^{L-1}\mathbb{E}_{X_{t_{q}}=\bx_{q}}\Big[ \Delta_{n,\ell}  f(u(X_{q},t_q))\Big]\Big]\Big\|_{\ell^\infty}. 
\end{aligned}
\end{equation*}    
\end{small}
By the Lipschitz continuity of $f(\cdot)$, the mean value theorem in time, and the Cauchy-Schwarz inequality, we obtain
\begin{eqnarray*}
&&B_{1,1}\leq \Big\| \mathbb{E}\Big[\sum_{q=0}^{L-1}\mathbb{E}_{X_q=\bx_{q}}\Big[\int_{t_{q}}^{t_{q+1}}\big|f(u(X_s,s))-f(u(X_s,t_{q}))\big| \,{\rm d}s\Big]\Big]\Big\|_{\ell^\infty},\\[4pt]
&&\hspace{23pt}\leq c\Big\| \mathbb{E}\Big[\sum_{q=0}^{L-1}\mathbb{E}_{X_q=\bx_{q}}\Big[\int_{t_{q}}^{t_{q+1}}\big|u(X_s,s)-u(X_s,t_{q})\big| \,{\rm d}s\Big]\Big]\Big\|_{\ell^\infty},\\[4pt]
&&\hspace{23pt}\leq c\|u_t\|_{L^\infty(\T_n;L^{\infty}(\Omega))}\Big\| \mathbb{E}\Big[\sum_{q=0}^{L-1}\mathbb{E}_{X_q=\bx_{q}}\Big[\int_{t_{q}}^{t_{q+1}}\big(s-t_{q}) \,{\rm d}s\Big]\Big]\Big\|_{L^\infty},\\[4pt]
&&\hspace{23pt}\leq c\|u_t\|_{L^\infty(\T_n;L^{\infty}(\Omega))}\Big\| \mathbb{E}\Big[\sum_{q=0}^{L-1}(t_{q+1}-t_{q})^2 \,{\rm d}s\Big]\Big\|_{L^\infty},\\[4pt]
&&\hspace{23pt}\leq ch_n{\Delta_{n,\ell}} \|u_t\|_{L^\infty(\T_n;L^{\infty}(\Omega))}.
\end{eqnarray*}
With the sphere radius fixed as $r = \sqrt{2d\,\Delta_{n,\ell}}$,  
the expectation of the path integral from $X_{q}:=X_{t_{n,\ell,q}}$ to $X_{{q+1}}:=X_{t_{n,\ell,q+1}}$  
admits the following compact representation \cite[Lemmas~2.2-2.4]{Sheng2023}: 
\begin{equation}
\mathbb{E}_{X_{t_{q}}=\bx}\Big[\int_{t_{q}}^{t_{q+1}}\! f(X_s,t_{q})\,{\rm d}s\Big] = W(\bx)\,\mathbb{E}_{{Q}_r}\big[f(Y,t_{q})\big], 
\quad Y \in \mathbb{B}_r^d,
\end{equation}
where, for any $\bx \in \mathbb{B}_r^d$,
\begin{equation*}
Q_r(\bx,\by) = \frac{G_r(\bx,\by)}{W(\bx)},\;\;\;W(\bx) = \int_{\mathbb{B}_r^d} G_r(\bx,\by)\,{\rm d}\by, 
\end{equation*}
and the Green's function $G_r(\bx,\by)$ is given by
\begin{equation*} 
G_r(\bx,\by)= \begin{cases}
-\frac{(\min\{x,y\}-a)(b-\max\{x,y\})}{b-a},\;\;\;&d=1,\\[2pt]
\frac{1}{2\pi}\,\log\!\big(\frac{r}{|\bx-\by|}\big),\;\;\;&d=2,\\[8pt]%\;\;\;\by^\ast=\Big(\frac{r^2}{ |\by|^2}\Big)· \by, \\[6pt]
 \frac{\Gamma(d/2-1)}{4\pi^{d/2}}\big(|\bx-\by|^{2-d}- r^{2-d}\big),\;\;\;&d\ge 3.
\end{cases}\end{equation*}
Moreover, one can verify easily that the weight function $W(\bx)$ has the following bound (cf. \cite[Prop.\,1.8]{port2012brownian}):
\begin{equation}\label{zeta_x}
W(\bx) =  \int_{\mathbb{B}_r^d}G_r(\bx,\by){\rm d}\by \leq c\,r^2,
\end{equation}
where the positive constant $c$ depends only on $d$.
Therefore, by the Lipschitz continuity of $f(\cdot)$ and \eqref{zeta_x}, we find that
\begin{eqnarray*}
&&B_{1,2}= \Big\| \mathbb{E}\Big[\sum_{q=0}^{L-1}\mathbb{E}_{X_q=\bx_{q}}\Big[\int_{t_{q}}^{t_{q+1}}f(u(X_s,t_{q})) \,{\rm d}s - \Big(\int_{t_{q}}^{t_{q+1}} \,{\rm d}s\Big) f(u(X_{q},t_q))\Big]\Big]\Big\|_{\ell^\infty}\\[4pt]
&&\hspace{21pt}= \Big\| \mathbb{E}\Big[\sum_{q=0}^{L-1}W(X_{q})\mathbb{E}_{Q_{r}}\big[f(u(Y_q,t_{q}))-f(u(X_{q},t_q))\big]\Big]\Big\|_{\ell^\infty}\\[4pt]
&&\hspace{21pt}\leq cN_\ast r^2 \Big\|\mathbb{E}_{Q_r}\Big[|u(Y_q,t_{q})-u(X_q,t_{q})|\Big] \Big\|_{\ell^\infty}.
\end{eqnarray*}
Using the the relations $r = \sqrt{2d\,\Delta_{n,\ell}}$, $N_\ast\Delta_{n,\ell} \leq h_n$ and \eqref{zeta_x}, we arrive that
\begin{eqnarray*} 
%\begin{aligned}
&&B_{1,2} %&\leq CN_\ast r^2 \Big\|\mathbb{E}_{Q_r}\big[|u(Y_q,t_{q})-u(X_q,t_{q})|\big]\Big\|_{L^\infty}
 \leq cN_\ast r^2 \Big\|\mathbb{E}_{Q_r}\Big[\frac{|u(Y_q,t_{q})-u(X_q,t_{q})|}{|Y_q-X_q|}|Y_q-X_q|\Big]\Big\|_{\ell^\infty}\\[4pt]
&&\hspace{21pt} \leq cN_\ast r^2\|u\|_{L^\infty(\T_n;C^{0,1}(\Omega))}\max_{Y_q\in\mathbb{B}_r^d(\bx)} \Big\|\mathbb{E}_{Q_r}\big[|Y_q-X_q|\big]\Big\|_{L^\infty}\\[4pt]
&&\hspace{21pt} \leq cN_\ast r^2\|u\|_{L^\infty(\T_n;C^{0,1}(\Omega))} \big\|\mathbb{E}_{Q_r}\big[r\big]\big\|_{L^\infty}\\[4pt]
&&\hspace{21pt} \leq cN_\ast r^3\|u\|_{L^\infty(\T_n;C^{0,1}(\Omega))}\leq ch_n \sqrt{\Delta_{n,\ell}}\|u\|_{L^\infty(\T_n;C^{0,1}(\Omega))}.% \leq C\Delta t^{\frac{3}{2}}|u|_{C^1(\Omega)} .
%\end{aligned}
\end{eqnarray*}
%where in the last step we used the relations $r = \sqrt{2d\,\Delta_{n,\ell}}$ and $N_\ast\Delta_{n,\ell} \leq h_n$.  
This yields the asserted estimate \eqref{result_B_1}.
\end{proof}
 
\begin{lem}
Let $u(x,t)$ denote the solution of \eqref{ufg_0} and $u^{(k)}_\ast(x,t)$ the numerical solution produced by {\rm Algorithm~\ref{algo_2}}. Suppose that, for some $r \geq 1$, $\partial_x^{\,r-1}u(x,t)$ is absolutely continuous on $\Omega\times\T_n$, $\partial_x^{\,r}u$ has bounded variation ${\rm BV}[\partial_x^{\,r}u]<\infty$, and $u \in H^m(\T_n;L^2(\Omega))$ with $m \geq 1$, and further assume that $f(\cdot)$ satisfies the Lipschitz condition \eqref{Lipschitz con}. Then, there holds
\begin{equation}\label{result_B_2}
B_2 \leq cN_x^{\frac12}h_n^{m+\frac12} M_n^{\frac12-m}\|\partial_t^mu\|_{L^2(\T_n;L^2(\Omega))}+ ch_nN_x^{\frac12-r},
\end{equation}
and 
\begin{equation}\label{result_B_3}
B_3 \leq ch_nN_x^{\frac{1}{2}}M_n^{\frac{1}{2}}E_{n,k-1}^{\infty},
\end{equation}	
where $c$ is a constant independent of $h_n,N_x,$ and $M_n$.
\end{lem}
\begin{proof}
Following the proof of Lemma~\ref{lemma4.2}, we fix the sphere radius $r=\sqrt{2d\,\Delta_{n,\ell}}$ and, for notational convenience, set $t_q:=t_{n,\ell,q}$ and $X_q:=X_{t_{n,\ell,q}}$ (hence $X_{q+1}:=X_{t_{n,\ell,q+1}}$). Then, by \eqref{indexL}, \eqref{zeta_x}, Lipschitz condition \eqref{Lipschitz con}, and Lemma \ref{lem3.1}, we have
\begin{eqnarray*}%\label{B_2}
%\begin{aligned}
&&B_2=\Big\|\mathbb{E}\Big[\mathbb{E}_{X_{t_{n-1}}=\bx}\Big[\mathcal{Q}_L\big[f(u)-f(I_{N_x}I_{M_n}^t u)\big]\Big]\Big]\Big\|_{\ell^\infty} \\[4pt]
&&= \Big\| \mathbb{E}\Big[\sum_{q=0}^{L-1}\mathbb{E}_{X_q=\bx_{q}}\Big[\Delta_{n,\ell}  \big[f(u(X_{q},t_q))-f(I_{N_x}I_{M_n}^t u(X_{q},t_q))\big]\Big]\Big]\Big\|_{\ell^\infty}\\[4pt]
&&= \Big\| \mathbb{E}\Big[\sum_{q=0}^{L-1}\mathbb{E}_{X_q=\bx_{q}}\Big[\Big(\int_{t_{n,\ell,q}}^{t_{n,\ell,q+1}} \,{\rm d}s\Big)\Big(f(u(X_{q},t_q))-f(I_{N_x}I_{M_n}^t u(X_{q},t_q))\Big)\Big]\Big]\Big\|_{\ell^\infty}\\[4pt]
&&= \Big\| \mathbb{E}\Big[\sum_{q=0}^{L-1}W(X_{q})\mathbb{E}_{Q_{r}}\big[f(u(X_{q},t_q))-f(I_{N_x}I_{M_n}^t u(X_{q},t_q))\big]\Big]\Big\|_{\ell^\infty}\\[4pt]
&& \leq cN_\ast r^2 \Big\|\mathbb{E}_{Q_r}\Big[\big|u(X_{q},t_q)-I_{N_x}I_{M_n}^t u(X_{q},t_q)\big|\Big]\Big\|_{\ell^\infty} 
\\&&\leq ch_n \big\|u-I_{N_x}I_{M_n}^t u\big\|_{L^\infty}\leq cN_x^{\frac12}h_n^{m+\frac12} M_n^{\frac12-m}\big\|\partial_t^mu\big\|_{L^2(\T_n;L^2(\Omega))}+ c h_nN_x^{\frac12-r}
\end{eqnarray*} 
Similarly, we derive from Lemma \ref{lem3.2} that
\begin{eqnarray*}%\label{B_2}
%\begin{aligned}
&&B_3=\Big\|\mathbb{E}\Big[\mathbb{E}_{X_{t_{n-1}}=\bx}\Big[\mathcal{Q}_L\big[f(I_{N_x}I_{M_n}^t u)-f(u_\ast^{n,(k-1)})\big]\Big]\Big]\Big\|_{\ell^\infty} \\[4pt]
&&\hspace{15pt}= \Big\| \mathbb{E}\Big[\sum_{q=0}^{L-1}\mathbb{E}_{X_q=\bx_{q}}\Big[\Delta_{n,\ell}  \big[f(I_{N_x}I_{M_n}^t u)-f(u_\ast^{n,(k-1)})\big]\Big]\Big]\Big\|_{\ell^\infty}\\[4pt]
&&\hspace{15pt}= \Big\| \mathbb{E}\Big[\sum_{q=0}^{L-1}\mathbb{E}_{X_q=\bx_{q}}\Big[\Big(\int_{t_{n,\ell,q}}^{t_{n,\ell,q+1}} \,{\rm d}s\Big)\Big(f(I_{N_x}I_{M_n}^t u)-f(u_\ast^{n,(k-1)})\Big)\Big]\Big]\Big\|_{\ell^\infty}\\[4pt]
&&\hspace{15pt}= \Big\| \mathbb{E}\Big[\sum_{q=0}^{L-1}W(X_{q})\mathbb{E}_{Q_{r}}\Big[f(I_{N_x}I_{M_n}^t u)-f(u_\ast^{n,(k-1)})\Big]\Big]\Big\|_{\ell^\infty}\\[4pt]
%&&\leq CN_\ast r^2 \Big\|\mathbb{E}_{Q_r}\Big[\frac{f(I_{N_x}I_{M_n}^t u)-f(u_\ast^{n,(k-1)})}{|I_{N_x}I_{M_n}^t u-u_\ast^{n,(k-1)}|}|I_{N_x}I_{M_n}^t u-u_\ast^{n,(k-1)}|\Big] \Big\|_{L^\infty}.
&&\hspace{15pt}\leq cN_\ast r^2 \Big\|\mathbb{E}_{Q_r}\Big[|I_{N_x}I_{M_n}^t u-u_\ast^{n,(k-1)}|\Big]\Big\|_{L^\infty} \leq ch_n \big\|I_{N_x}I_{M_n}^t(u-u_\ast^{n,(k-1)})\big\|_{L^\infty}
\\&&\hspace{15pt}\leq ch_n N_x^{\frac{1}{2}} M_n^{\frac{1}{2}}\big\|u-u_\ast^{n,(k-1)}\big\|_{L^\infty}.
\end{eqnarray*} 
This completes the proof.
\end{proof}

We now state the convergence result for the proposed iteration scheme, which reads as
\setcounter{theorem}{0}
\begin{theorem}\label{thm4.1}
Let $u(x,t)$ denote the solution of \eqref{ufg_0} and $u^{(k)}_\ast(x,t)$ the numerical solution generated by {\rm Algorithm~\ref{algo_2}}. Suppose that, for some $r \geq 1$, $\partial_x^{\,r-1}u(x,t)$ is absolutely continuous on $\Omega\times\T_n$, $\partial_x^{\,r}u$ has bounded variation ${\rm BV}[\partial_x^{\,r}u]<\infty$, $u_t\in L^\infty(\T_n;L^{\infty}(\Omega))$, and $u \in H^m(\T_n;L^2(\Omega))\cup L^\infty(\T_n;C^{0,1}(\Omega))$ with $m \geq 1$. Further assume that $f(\cdot)$ satisfies the Lipschitz condition \eqref{Lipschitz con}. Then, there holds
\begin{equation}\label{eqn:thm4.1_1}
\begin{aligned}
E_{n,k}^{\infty} &\leq \rho E_{n,k-1}^{\infty}+ch_n\Delta_{n,\ell}\|u_t\|_{L^\infty(\T_n;L^{\infty}(\Omega))}+ch_n \sqrt{\Delta_{n,\ell}}\|u\|_{L^\infty(\T_n;C^{0,1}(\Omega))} \\[4pt]
&\quad + cN_x^{\frac12}h_n^{m+\frac12} M_n^{\frac12-m}\big\|\partial_t^mu\big\|_{L^2(\T_n;L^2(\Omega))}+ ch_nN_x^{\frac12-r}, 
\end{aligned}
\end{equation}    
where the positive constant  $c$ is independent of $N_x$, $M_n$, $h_n$, $\Delta_{n,\ell}$, $k$,  and $\rho = ch_nN_x^{\frac{1}{2}}(M_n+1)^{\frac{1}{2}}.$ 
By choosing $c$, $h_n$, $N_x$, and $M_n$ such that $\rho < 1$, the sequence $\{E^\infty_{n,k}\}_{k\ge 0}$ converges geometrically at rate $\widetilde{\rho}$, until it reaches a threshold given by
\begin{equation}\label{eqn:thm4.1_2}
\begin{aligned}
E_{n,k}^{\infty} &\leq \frac{1}{1-\rho} \Big[ ch_n\Delta_{n,\ell}\|u_t\|_{L^\infty(\T_n;L^{\infty}(\Omega))}+ch_n \sqrt{\Delta_{n,\ell}}\|u\|_{L^\infty(\T_n;C^{0,1}(\Omega))} \\[4pt]
&\quad + cN_x^{\frac12}h_n^{m+\frac12} M_n^{\frac12-m}\big\|\partial_t^mu\big\|_{L^2(\T_n;L^2(\Omega))}+ ch_nN_x^{\frac12-r} \Big].    
\end{aligned}
\end{equation}     
\end{theorem}
\begin{proof}
According to the definition of $E_{n,k}^\infty$, we have
\begin{equation*} 
B_4=\Big\|\mathbb{E}\Big[\mathbb{E}_{X_{t_{n-1}}=\bm{x}}\Big[ u^{n-1}(X_t,t_{n-1})\mathbb{I}_{\tau_{\Omega}^{\bx}>t}-u_\ast^{n-1}(X_t,t_{n-1})\mathbb{I}_{\tau_{\Omega}^{\bx}>t}\Big]\Big]\Big\|_{\ell^\infty}\leq E_{n-1}^\infty,
\end{equation*}
which can be neglected, as the iteration at the previous time step $t_{n-1}$ has already achieved the target accuracy, namely the machine precision, at the grid nodes, and the temporal grid points are chosen to be Lobatto points, which naturally include $t_{n-1}$.   
A combination of \eqref{result_B_1},\eqref{result_B_2}, and \eqref{result_B_3}  leads to
\begin{equation*}
\begin{aligned}
E_{n,k}^{\infty} 
&\leq ch_nN_x^{\frac{1}{2}}M_n^{\frac{1}{2}} E_{n,k-1}^{\infty}+ch_n\Delta_{n,\ell}\|u_t\|_{L^\infty(\T_n;L^{\infty}(\Omega))}+ch_n \sqrt{\Delta_{n,\ell}}\|u\|_{L^\infty(\T_n;C^{0,1}(\Omega))} \\[4pt]
&\quad + cN_x^{\frac12}h_n^{m+\frac12} M_n^{\frac12-m}\big\|\partial_t^mu\big\|_{L^2(\T_n;L^2(\Omega))}+ ch_nN_x^{\frac12-r} .
\end{aligned}
\end{equation*}
This ends the proof.
\end{proof}	

\section{Numerical results} \label{sect4}\setcounter{lem}{0} \setcounter{thm}{0}  \setcounter{rem}{0} 	
In this section, we present several numerical results to demonstrate the effectiveness and accuracy of the time stepping spectral Monte Carlo methods for solving the semilinear parabolic equation~\eqref{ufg_0}.  
More importantly, we highlight several advantages of the $hp$-version time-stepping approach, in particular its suitability for long-time simulations and its ability to handle problems with initial singularity.  
In practical computations, we adopt the following discrete~$L^\infty$-error to quantify the error between the exact solution~$u(\bx,t)$~and its numerical approximation:
$$
\text{Error} = \max_{j,n,\ell} \Big|\mathbb{E}\big(u({\bx}_j,{t}_{n,\ell}) - u_\ast({\bx}_j,{t}_{n,\ell})\big)\Big|,
$$
where $u_\ast$ denotes the numerical solution, and $\{{\bx}_j\}_{j=0}^{N_x}$ and $\{{t}_{n,\ell}\}_{\ell=0}^{M_n}$ denote the quadrature nodes over the interval~$\Omega$~and the subinterval~$\T_n$, respectively.

%% -------------------------- example 1 ----------------------------
\begin{exa}\label{Ex:1}{\bf (Accuracy test for the single-step method)} 
We first consider the nonlinear equation~\eqref{ufg_0} in one-dimensional case with the homogeneous boundary condition $g(x,t)=0$ on $\partial\Omega$, where we choose the exact solution
\begin{equation}\label{solu1}
u(x,t)=\sin(\pi x)\cos(t), \quad (x,t)\in(0,1)\times(0,T],    
\end{equation}
together with the nonlinearity $f(u)=u\sin(u)$, and add the following additional source term
 $$
f_{\rm add}(x,t)=-\sin(\pi x)\sin(t)+\pi^2\sin(\pi x)\cos(t)+\sin(\pi x)\cos(t)\sin\!\big(\sin(\pi x)\cos(t)\big),
$$
so that the equation is exactly satisfied.
\end{exa}
\begin{figure}[htbp]
\centering
\hspace{-10pt}
\subfigure{
\includegraphics[width=0.335\textwidth]{ 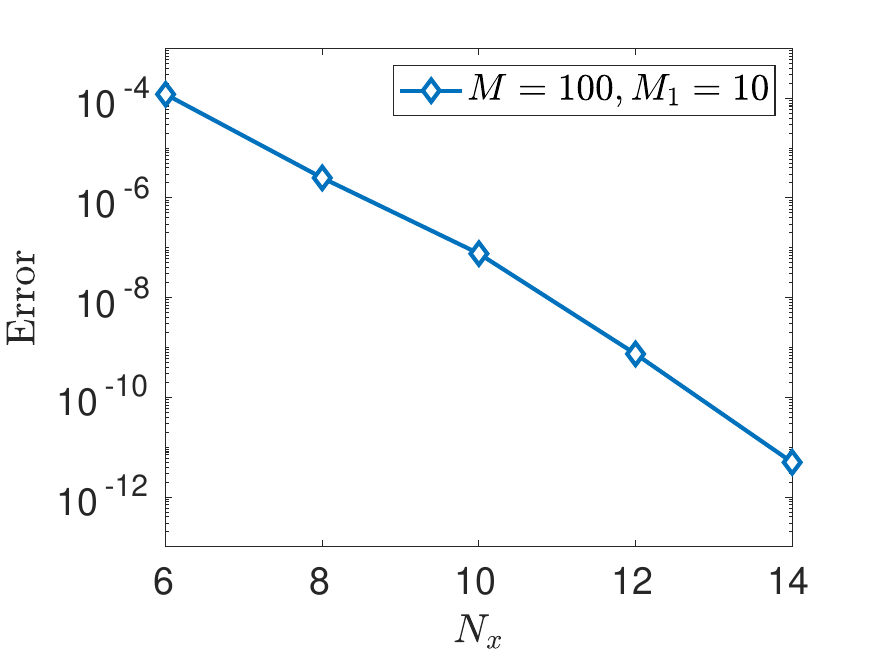}}\hspace{-10pt}
\subfigure{
\includegraphics[width=0.335\textwidth]{ 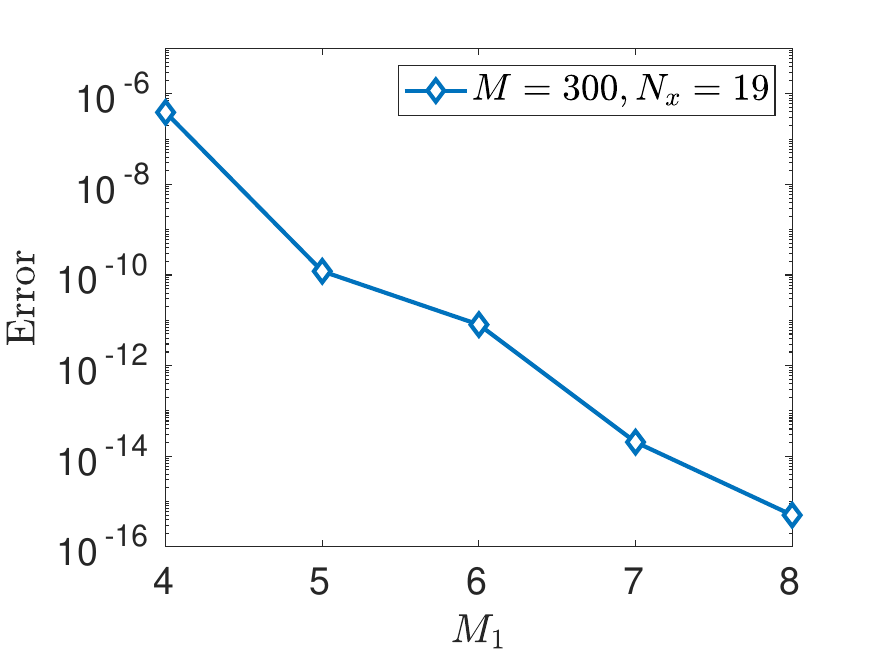}}\hspace{-10pt}
\subfigure{
\includegraphics[width=0.335\textwidth]{ 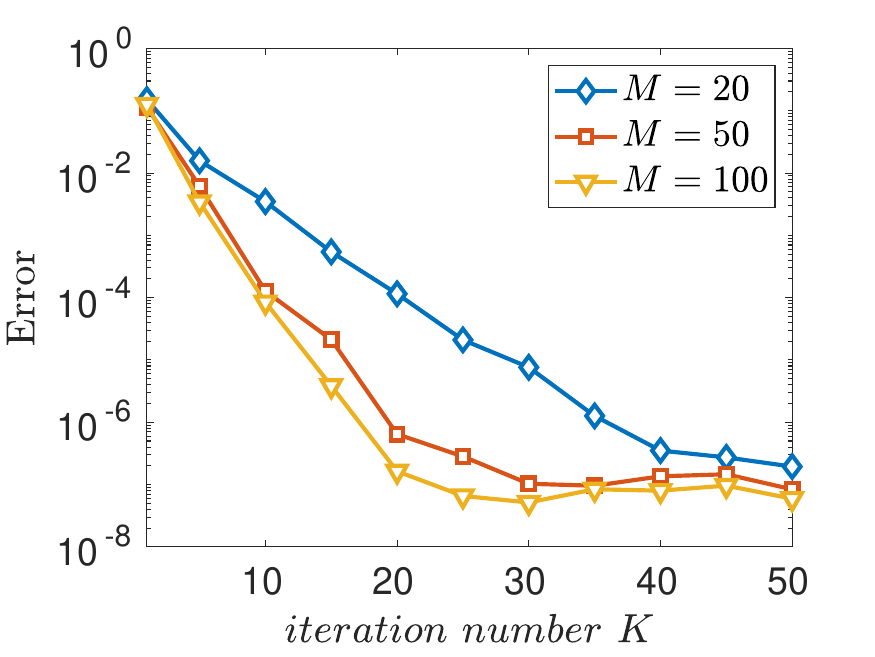}}
\caption{The discrete maximum errors of $u(x,t)$ in~\eqref{solu1} with the final time $T=0.5$. Left: spatial convergence with respect to the polynomial degree $N_x$; Middle: temporal convergence with respect to the polynomial degree $M_1$; Right: iterative convergence with respect to the iteration count $K$ for varying numbers of simulation paths $M$, with $N_x=10$ and $M_1=6$ fixed.} 
\label{fig:EX1}
\end{figure}

We first verify that the residual-iteration-based stochastic algorithm~\eqref{iterscheme} achieves spectral accuracy at the Gauss points.  
To this end, we begin with a single-step spectral Monte Carlo method for solving \eqref{solu1}, i.e., $N_h=1$, with the final time set to $T=0.5$.  
More specifically, for the reconstruction of the numerical solution, we employ Jacobi-Gauss (JG) quadrature points in the spatial direction, while shifted Legendre-Gauss-Lobatto (LGL) nodes are used for temporal interpolation.  
As shown in Fig.\,\ref{fig:EX1}, the numerical solution achieves very high accuracy as the resolution parameters increase, which is consistent with the theoretical expectations for spectral methods, even when dealing with nonlinear solutions.  
In more detail, Fig.\,\ref{fig:EX1}\,(left) plots the discrete maximum error of $u(x,t)$ in~\eqref{solu1} against different values of $N_x$ in semi-log scale, where the number of simulation paths is fixed at $M=100$, and the polynomial degree of LGR interpolation is set to $M_1=6$.  
Similarly, Fig.\,\ref{fig:EX1}\,(middle) shows the discrete maximum errors against different $M_1$ in semi-log scale, where the number of simulation paths is fixed at $M=100$ and the polynomial degree of JG interpolation is set to $N_x=19$.  
It is worth noting that the required iteration count $K$ is inversely correlated with the sample size $M$: as illustrated in Fig.\,\ref{fig:EX1}\,(right), spectral accuracy can be achieved with fewer iterations as $M$ increases.

%% -------------------------- example 2 ----------------------------
\begin{exa}\label{Ex:3}{\bf (Long time simulation)} 
To examine the long-time stability of the proposed time-stepping spectral Monte Carlo method and to further demonstrate the applicability of our algorithm, we consider the following more general fractional Allen Cahn equation in 1D:
\begin{equation}\label{eqn:frac_dif}
\begin{cases}
\partial_t u+(-\Delta)^{\frac{\alpha}{2}} u+u^3-u=0,\;&(x,t) \in \ \Omega \times (0,\infty) ,\\[4pt]
u(x,t)=g(x,t),\quad & (x,t) \in \ \Omega^c \times (0,\infty) ,\\[4pt]
u(x,0)=u_0(x),\quad &x \in \Omega,
\end{cases}
\end{equation}
where $(-\Delta)^{\frac{\alpha}{2}} $ is defined via the hypersingular integral representation:
\begin{equation*}
(-\Delta)^{\frac{\alpha}2} u(\bx)=C_{d,\alpha}\, {\rm p.v.}\! \int_{\mathbb R^d} \frac{u(\bx)-u(\by)}{|\bx-\by|^{d+\alpha}}\, {\rm d}\by,\quad
C_{d,\alpha} :=\frac{\alpha2^{\alpha}\Gamma(\frac{d+\alpha}{2})}{\pi^{\frac{d}2}\Gamma(1-\frac{\alpha}2)},
\end{equation*}
with ``p.v." standing for the principal value. 
For convenience in verifying the accuracy, we similarly add an extra source term so that \eqref{eqn:frac_dif} admits the following manufactured exact solution
 $u(x,t) = (1-x^2)_+^\frac{\alpha}{2}\sin(x)\frac{{\rm exp}(\cos(t))}{1+t^2},$ with $a_+=\max(0,a)$.
\end{exa}

\begin{figure}[htbp]
\label{fig:EX2}
\centering  \hspace{-12pt}
\subfigure{ 
\includegraphics[width=6.5cm,height=5.5cm]{ 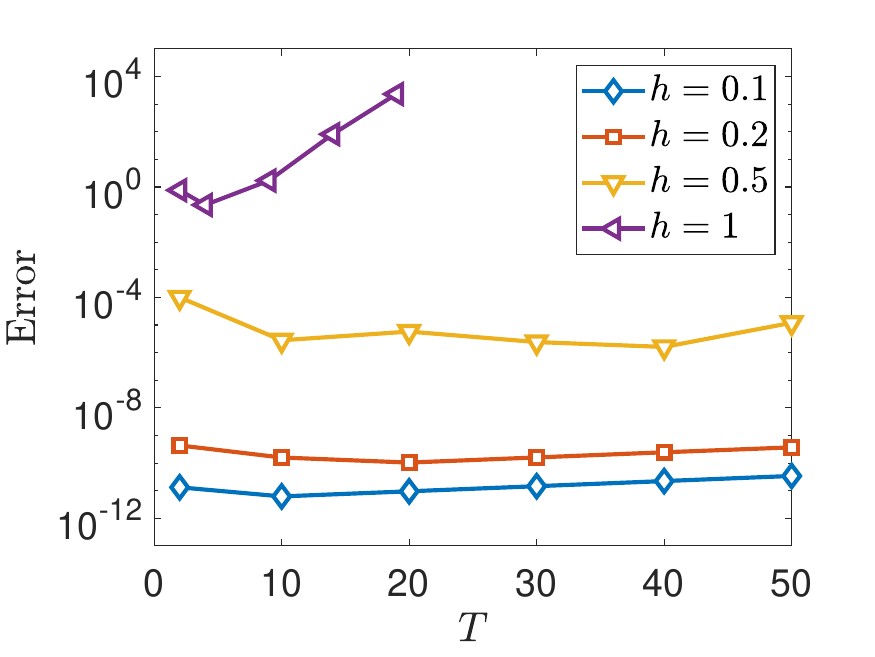}}\hspace{-20pt}
\subfigure{ 
\includegraphics[width=6.5cm,height=5.5cm]{ 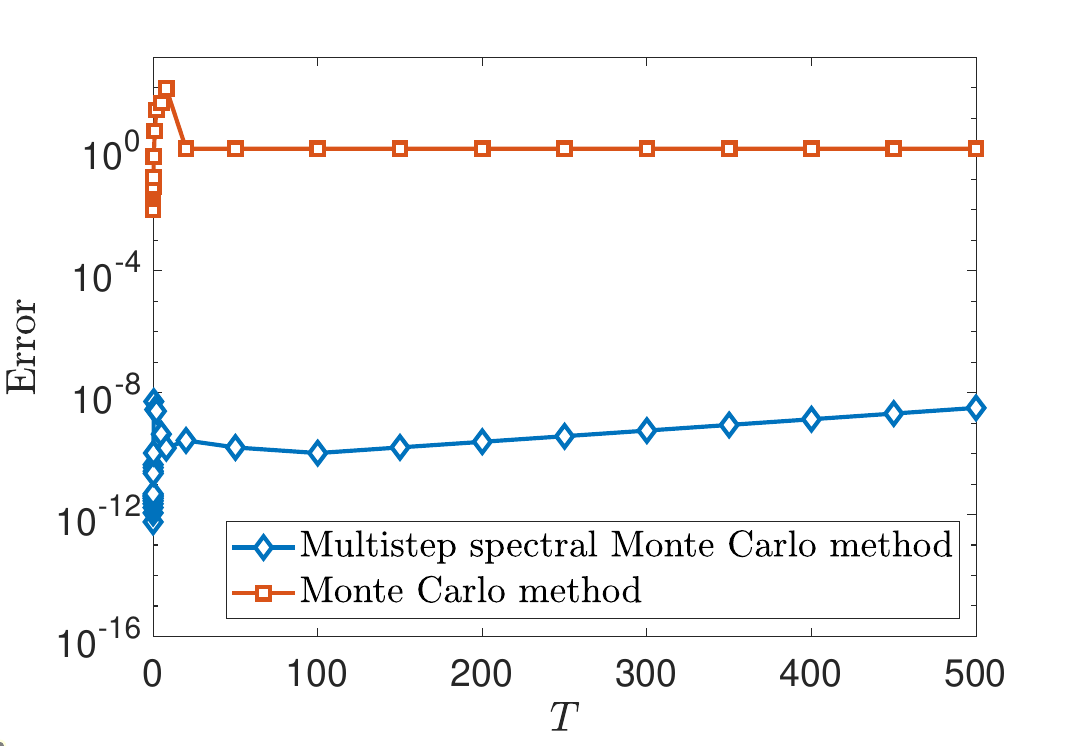}}
\caption{The discrete maximum errors against $T$ for the long-time simulation of \eqref{eqn:frac_dif}. Left: errors versus various time-step sizes $h$ with $T=50$;  
Right: comparison between the classical Monte Carlo method and the $hp$-version time-stepping method for a long-time simulation with $T=500$.}
\end{figure}

The proposed method introduced in Section~\ref{sect2} can be directly extended to the fractional case~\eqref{eqn:frac_dif}, with the only modification being that the underlying stochastic process is replaced by an $\alpha$-stable Lévy process instead of the standard Brownian motion.  
Consequently, in~\eqref{Xi}, the jump length $J$ is replaced by that associated with the $\alpha$-stable Lévy process, which is given by (cf.\cite{Sheng2025})
$$
J := J(\omega; r_\ell, \alpha)
= \frac{r}{\sqrt{B\!\left(1-\frac{\alpha}{2}, \frac{\alpha}{2}\right)
- B^{-1}\!\left(\frac{\pi\,\omega}{\sin(\pi\alpha/2)}; 1-\frac{\alpha}{2}, \frac{\alpha}{2}\right)}},
\quad \omega \in (0,1),
$$
where the radius of the ball is chosen as  
$$
r = \big(\Delta t_n/{C}_{d}^\alpha\big)^{1/\alpha}, 
\quad \text{with} \quad
 {C}_{d}^\alpha = \frac{\Gamma\!\left(\frac{d}{2}\right)}
{2^{\alpha}\,\Gamma\!\left(1 + \frac{\alpha}{2}\right)\,\Gamma\!\left(\frac{d+\alpha}{2}\right)}.
$$
In addition, for the reconstruction of the solution, the parameters of the generalized Jacobi functions are adjusted from $(1,1)$ in the integer-order case to $(\tfrac{\alpha}{2},\tfrac{\alpha}{2})$ in the fractional case; see~\cite{Feng2025} for further details.

We employ the time-stepping spectral Monte Carlo method to numerically solve\\ \eqref{solu1}.  
To highlight the advantage of the multi-domain approach for long-time simulations, we first compare the influence of different time-step sizes on the accuracy over long time intervals.  
In Fig.~\ref{fig:EX2}\,(left), we present the discrete maximum errors of~\eqref{eqn:frac_dif} for $\alpha=0.4$, $M_n\equiv6$, $N_x=10$, $M=10$, and $T=50$, with various step sizes $h_n\equiv0.1,0.2,0.5,1$.  
To guarantee the desired accuracy of the iterative procedure within each subinterval~$\mathcal{T}_n$, the iteration is terminated once the difference between two successive iterates falls below the prescribed tolerance $10^{-10}$.  
In practice, the required number of iterations~$K$ is typically significantly smaller than the upper bound of $200$.
We observe that, consistent with traditional $hp$-type deterministic methods, the overall accuracy can be enhanced by refining the time-step size~$h$.
Moreover, to better illustrate the long-time stability of the proposed time-stepping spectral Monte Carlo method, we compare the error evolution of the classical Monte Carlo algorithm and the proposed scheme at $T=500$.  
As shown in Fig.~\ref{fig:EX2}\,(right), the error of the classical Monte Carlo method rapidly grows to the order of $\mathcal{O}(1)$, indicating a complete loss of accuracy.  
In contrast, the proposed algorithm maintains high accuracy and stability even for very long-time simulations, which highlights a key advantage of the $hp$-version approach.

%% -------------------------- example 3----------------------------
\begin{exa}\label{Ex:2}{\bf (Initial singularity)} Next, we demonstrate the capability of our method in handling problems with initial singularities.  
To this end, we consider the fractional Allen-Cahn equation~{\rm \eqref{eqn:frac_dif}}, which admits the following manufactured exact solution obtained by adding an appropriate source term:
\begin{equation}\label{eqn_EX2}
\begin{split}
u(x,t) = (1-x^2)_+^{\frac{\alpha}{2}}\frac{\cos(x^3+1)}{10x^2+1}\,t^r,
\end{split}
\end{equation}
subject to the homogeneous boundary condition $g(x,t)=0$ on $\Omega^c\times(0,T)$.  
Here, we set $T=0.5$ and choose a non-integer parameter $r>\tfrac{1}{2}$, so that the solution exhibits a weak singularity at $t=0$.
\end{exa}

We numerically solve the nonlinear parabolic equation~\eqref{ufg_0} with the exact solution~\eqref{eqn_EX2} using the algorithm described in Algorithm~\ref{algo_2}.  
For comparison, in the temporal reconstruction we first employ a fixed time-step size $h_n \equiv h$, and within each subinterval $\mathcal{T}_n$ the computation is carried out with a uniform polynomial degree (i.e., fixed degree of polynomials $M_n\equiv M$).  
Fig.~\ref{fig:EX2andEX3}\,(left) presents the maximum errors for the parameters $\alpha=0.4$, $T=1$, and $h_n\equiv 0.1$, with convergence rates tested at $r=0.51$, $1.51$, $2.51$, and $3.51$.  
The results clearly exhibit algebraic convergence behavior, which is consistent with the approximation properties of traditional spectral element methods on uniform mesh for functions with singularities.

To more effectively resolve the initial singular behavior of the solution, we employ geometrically refined meshes in time together with linearly increasing degree of polynomial .  
This strategy concentrates computational effort in regions of rapid variation or singularity, while systematically improving the approximation accuracy, thereby ensuring near-exponential convergence and computational efficiency. Specifically, we take  
\begin{itemize}
\item geometrically refined mesh:
$$
t_0=0,\quad t_n = T \times 0.2^{\,N-n},\;\;1\leq n\leq N;
$$
\item linearly increasing polynomial degrees:
$$
M_1=1,\quad M_n=\max\big\{1,\,[1.5n]\big\},\;\;2\leq n\leq N,
$$
where $[1.5n]$ denotes the greatest integer not exceeding $1.5n$.
\end{itemize}

In Fig.~\ref{fig:EX2andEX3}\,(right), we plot the discrete maximum errors for $\alpha=0.4$, $T=1$, and $h=0.1$ with regularity parameters $r=0.51$, $0.71$, $0.91$, and $1.11$.  
The results clearly exhibit the expected near-exponential convergence, achieving the same accuracy as classical $hp$-version deterministic methods.  
Moreover, we observe that the convergence slope increases as the regularity parameter $r$ increases, and it exhibits better convergence rates and higher accuracy compared to uniform meshes.
It is worth noting that, in practical computations, we adopt an adaptive strategy to determine the number of iterations $K$ within each subinterval $\mathcal{T}_n$, where the iteration is terminated once a prescribed tolerance $10^{-13}$ is reached.  
These numerical results demonstrate that the proposed $hp$-version spectral Monte Carlo method attains the same level of approximation accuracy as the classical $hp$-version spectral element method even for solutions with initial singularities.  
Importantly, the proposed method does not require solving linear systems of equations; instead, it benefits from the inherent parallelism of Monte Carlo sampling at each temporal and spatial node, which highlights the advantage of the $hp$-version Monte Carlo approach.

\begin{figure}[htbp]
\centering  \hspace{-8pt}
\subfigure{
\includegraphics[width=6.5cm,height=5.5cm]{ 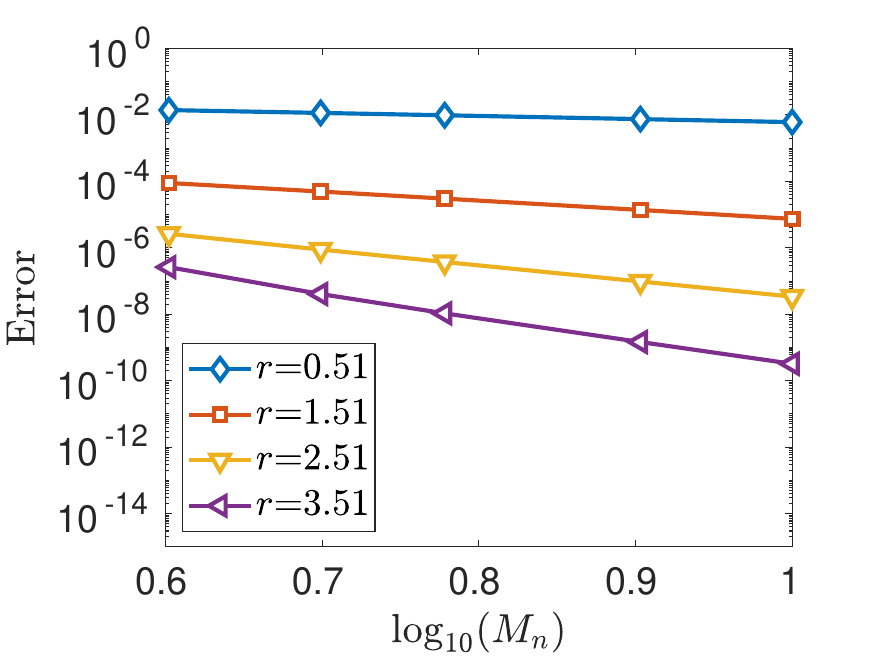}}\hspace{-20pt}
\subfigure{
% \label{fig:EX2_2}
\includegraphics[width=6.5cm,height=5.5cm]{ 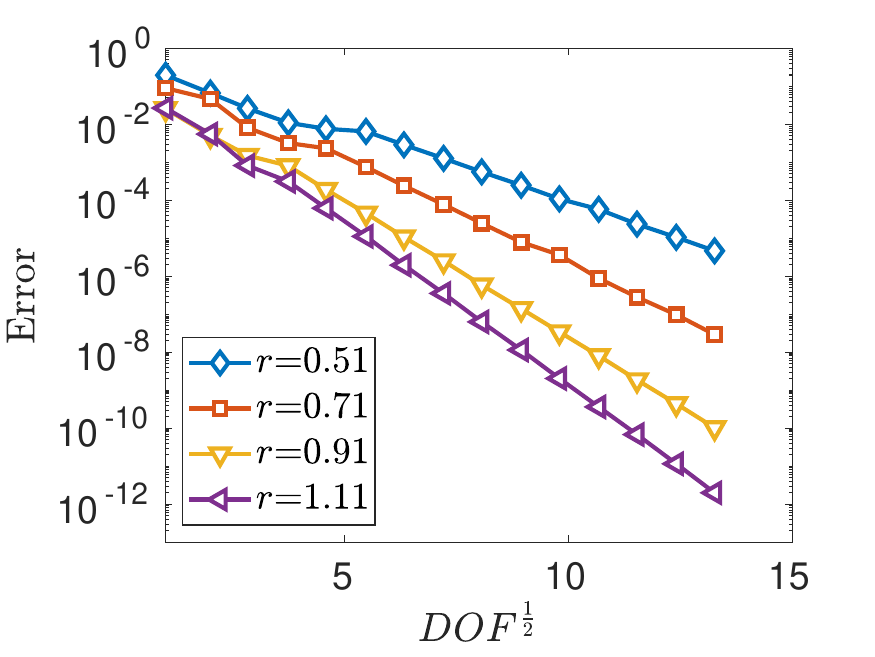}}
\caption{Discrete maximum errors for the initial singularity case~\eqref{eqn_EX2}.  
Left: errors against $\log_{10}(M_n)$ for various values of $r$ under $h_n\equiv0.1$ and the uniform mode $M_n$ ;  
Right: errors against ${\rm DoF^{1/2}}$ for various values of $r$ using a geometrically refined mesh and linearly increasing polynomial degrees $M_n$.}
\label{fig:EX2andEX3}
\end{figure}

%% -------------------------- example 4 ----------------------------
\begin{exa}\label{Ex:4}{\bf (Application to higher-dimensional problems)} 
We consider the nonlinear diffusion equation~\eqref{ufg_0} in higher dimensions, subject to the homogeneous Dirichlet boundary condition $g(\bx,t)=0$ on $\partial\Omega$.  
%A manufactured {\color{blue}nontensorial}exact solution is specified as  
 A manufactured nontensorial exact solution is prescribed as
\begin{small}
\begin{equation}\label{eqn_EX4} 
u(\bx,t)=\exp\bigg(\frac{\prod_{i=1}^dx_i}{8}\bigg)\prod_{i=1}^d(1-x_i^2)\cos(t),
\,\bx=(x_1,\dots,x_d)\in\Omega=(-1,1)^d,\,t\in(0,T],
\end{equation}    
\end{small}
where the nonlinear term is chosen as $f(u)=u^{3/2}$, and an appropriate source term is introduced to ensure that the prescribed solution~\eqref{eqn_EX4} satisfies the governing equation exactly.
\end{exa}
\begin{figure}[htbp]
\centering   \hspace{-8pt}
\subfigure{
\label{fig:EX4_1}
\includegraphics[width=6.5cm,height=5.5cm]{ 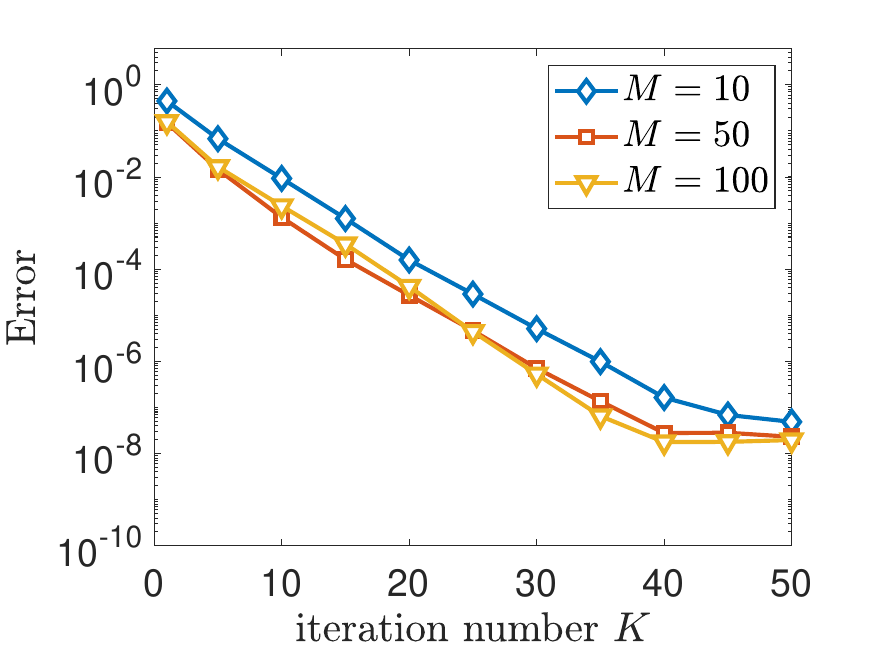}}  \hspace{-20pt}
\subfigure{
\label{fig:EX4_2}
\includegraphics[width=6.5cm,height=5.5cm]{ 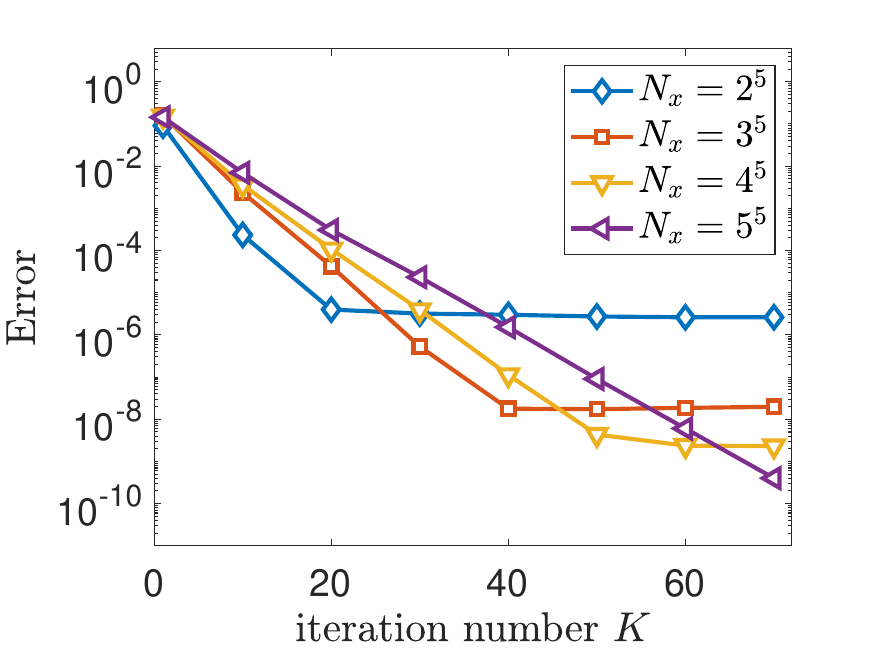}}
\caption{Discrete maximum errors against iteration number for the 5D case~\eqref{eqn_EX4}.  Left: errors versus various path numbers $M$ with fixed $N_x=3^5$ and $M_n=6$;  
Right: errors versus various $N_x$ with fixed $M=50$ and $M_n=6$.}
\label{fig:EX4}
\end{figure}

In this example, we consider the case $d=5$ with $T=0.2$, employing Jacobi-Gauss quadrature points along each spatial direction $x_i$ $(i=1,2,\dots,5)$ for discretization, together with shifted Legendre-Gauss-Lobatto nodes for temporal discretization.
Fig.\,\ref{fig:EX4} clearly illustrates that the numerical error can be significantly reduced to very high accuracy as the number of iterations increases.  
More specifically, Fig.\,\ref{fig:EX4}\,(left) depicts the discrete maximum error with fixed $N_x=3^5$ and $M_n=6$, showing that for different numbers of sample paths $M$, the error rapidly converges to a high level of accuracy. 
These experiments demonstrate that optimizing the number of sample paths $M$ can effectively improve the accuracy and thus reduce the required iteration count $L$ for achieving a target precision. Moreover, Fig.\,\ref{fig:EX4}\,(right) plots the discrete maximum error with fixed $M=50$ and $M_n=6$ while varying $N_x$, which shows that the achievable accuracy strongly depends on the polynomial degree employed in the reconstruction.
In addition, maintaining an appropriate scaling relationship between $M$ and the discretization parameters $N_x$ and $M_n$ is crucial for ensuring the convergence of the algorithm.  
In particular, finer spatial-temporal resolutions require larger values of $M$.

\begin{exa}\label{Ex:5}{\bf (Time evolutions on irregular domain)} 
We consider the nonlinear equation \eqref{ufg_0} within the 
% star-shaped 
irregular domain, where the corresponding computational domain $\Omega$ can be determined by the following polar coordinate transformation of form
\begin{equation}\label{polarcoor}
\begin{cases}
	{x=r R(\theta)\cos(\theta),} & {(r,\theta)\in(0,1)\times(0,2\pi),} \\[4pt]
	{y=r R(\theta)\sin(\theta),} & {(r,\theta)\in(0,1)\times(0,2\pi).}
\end{cases}
\end{equation}
\end{exa}

\underline{\textbf{(I)  Deterministic phase field evolution in star-shaped domain.}} In \\this example, we take $R(\theta)=7+2\sin(3\theta)$, $\theta\in[0,2\pi]$. The equation is subject to homogeneous Dirichlet boundary conditions $g(\bx,t) = 0$ on $\partial \Omega$ and involves a standard Allen–Cahn type nonlinearity $f(u) = u - u^3$. The initial condition is prescribed as
\begin{equation}\label{ACinitial}
u_0(\bx) = \sum_{i=1}^{3} c_i \exp\Big(-\frac{|\bx - \bx_i|^2}{\delta}\Big), \quad \bx = (x_1, x_2) \in \Omega,
\end{equation}
where $\delta=1$ controls the local structure of the initial value,  and the peak locations and coefficients are given by
\begin{equation}\label{points}
\bx_1 = (3.5,2), \quad \bx_2 = (-3.5, 2), \quad \bx_3 = (0,-3),
\end{equation}
and $c_1=c_2=1$, $c_3=2$, so that the initial solution presents the form of four local Gaussian distributions.

In this case, the focus is not on achieving spectral accuracy but on using the spectral Monte Carlo method to enhance the resolution of the numerical solution. For accurate function reconstruction over complex geometries, we adopt a mapped spectral technique that transfers quadrature and interpolation nodes from a reference rectangle to the physical domain via a polar‐type transformation \eqref{polarcoor}. This enables high‐order reconstruction from coarse data and explicit evaluation of its Laplacian in physical coordinates. As these reconstruction procedures, though important, are ancillary to the central topic, we omit the details and refer to \cite{Wang2023}.

The initial condition in \eqref{ACinitial} comprises three localized Gaussian peaks centered at the locations specified in \eqref{points}, embedded within a low‐valued background, thereby creating a strongly inhomogeneous initial state. Fig.\,\ref{Ex6} depicts the time evolution of the Allen–Cahn equation on the prescribed domain. In the early stages (Figs.~\ref{Ex6_t02}--\ref{Ex6_t06}), the peaks diffuse and interact under the competing influences of the reaction and diffusion terms, driven by the underlying energy‐minimization principle. By $T=1$ (Fig.~\ref{Ex6_t1}), the initially distinct peaks have largely merged into a single connected structure near the domain’s center. As time advances (Figs.~\ref{Ex6_t15} and \ref{Ex6_t2}), the solution becomes progressively smoother and more stable, approaching a metastable or steady‐state configuration. This behavior accords with the well‐known dynamics of the Allen-Cahn equation, in which interfaces evolve to reduce the total interfacial energy. The results clearly illustrate the model’s intrinsic coarsening mechanism and energy‐dissipation property.

\begin{figure}[ht!]
\centering \hspace{-16pt}%\vspace{-11pt}
\subfigure[T=0.2]{\label{Ex6_t02} 
\includegraphics[width=4.2cm,height=3.9cm]{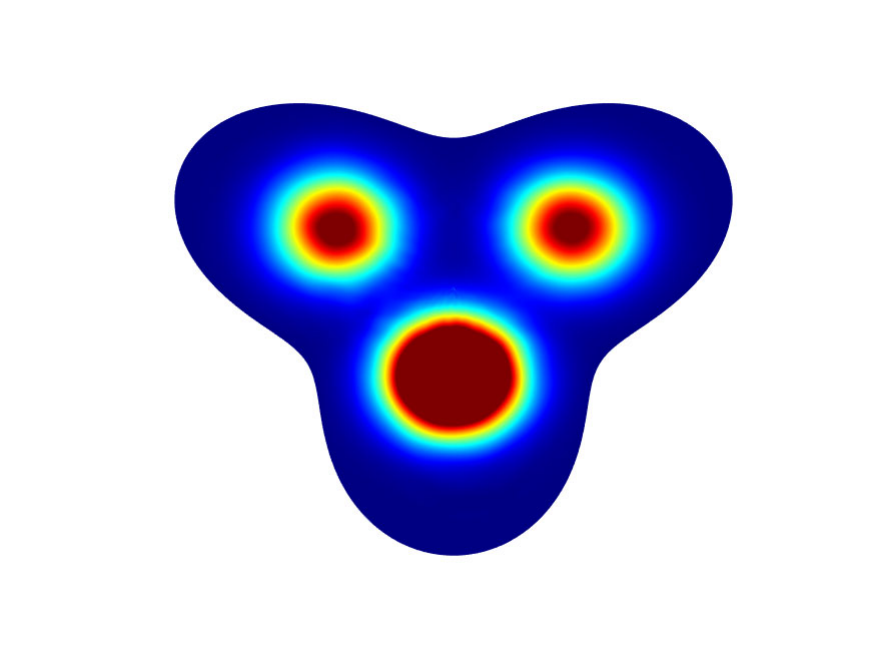}}\hspace{-10pt}
\subfigure[T=0.4]{\label{Ex6_t04} 
\includegraphics[width=4.2cm,height=3.9cm]{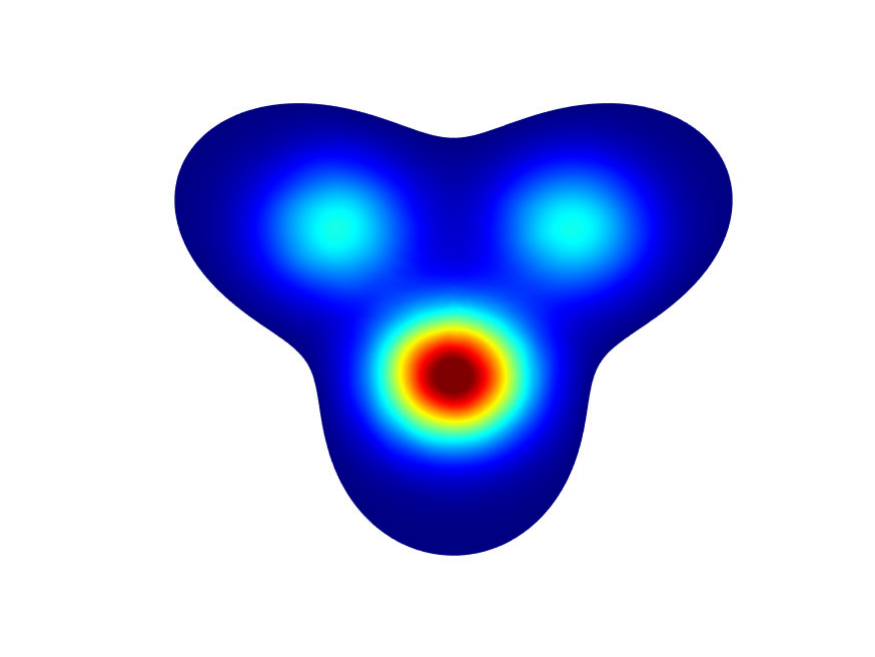}}\hspace{-10pt}
\subfigure[T=0.6]{\label{Ex6_t06} 
\includegraphics[width=4.6cm,height=3.9cm]{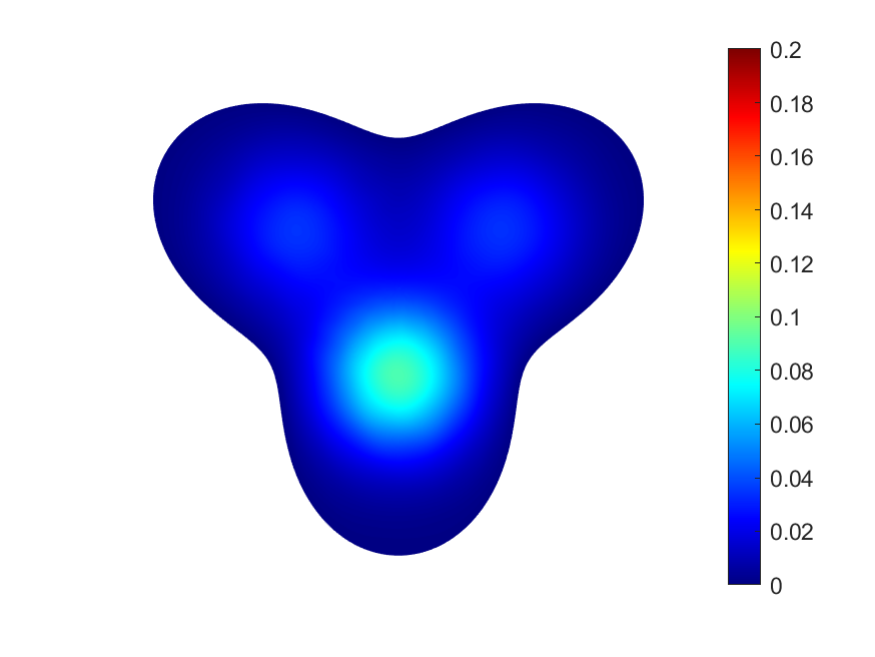}}\hspace{-16pt}\vspace{-11pt}
\subfigure[T=1]{\label{Ex6_t1} 
\includegraphics[width=4.2cm,height=3.9cm]{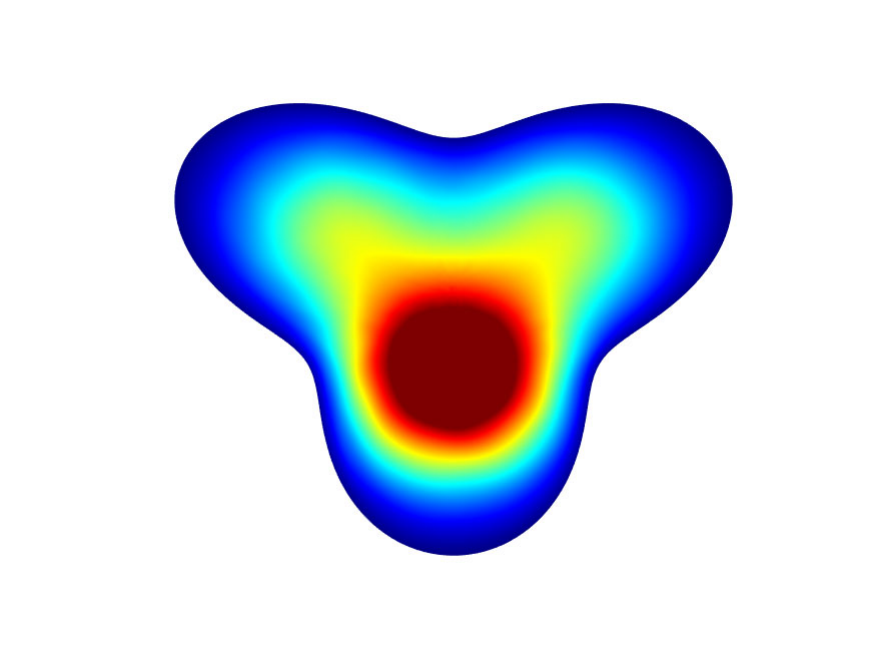}}\hspace{-10pt}
\subfigure[T=1.5]{\label{Ex6_t15}
\includegraphics[width=4.2cm,height=3.9cm]{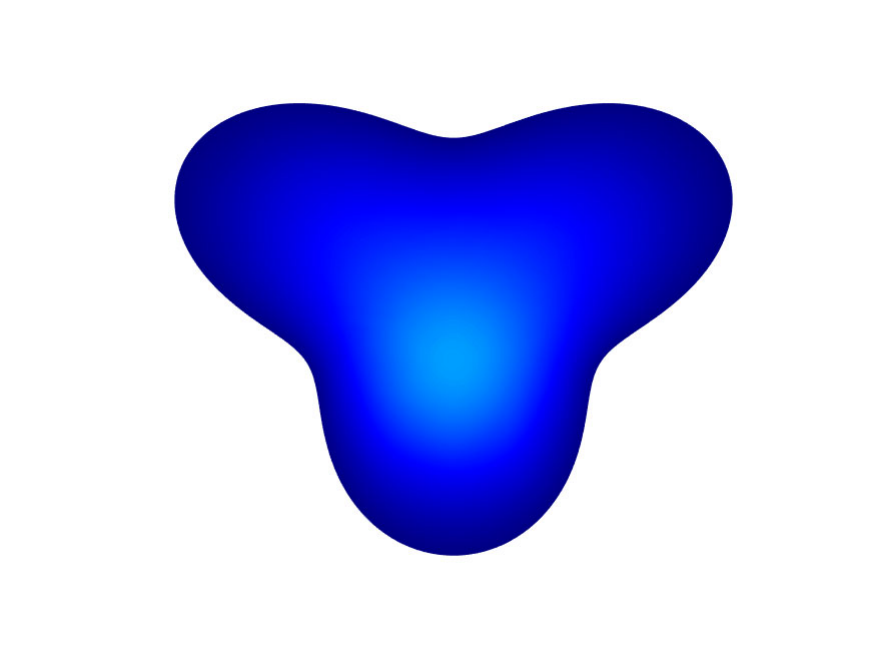}}\hspace{-10pt}
\subfigure[T=2]{\label{Ex6_t2}
\includegraphics[width=4.6cm,height=3.9cm]{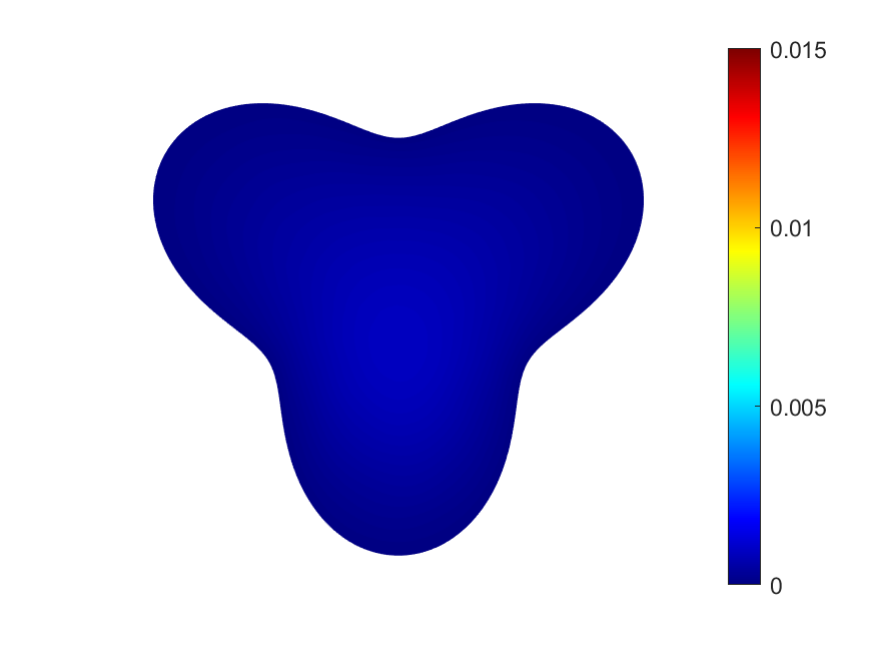}}
\vspace{-2pt}\caption{Spinoidal decomposition governed by the Allen-Cahn equation. Simulations are obtained by Algorithm \ref{algo_2} with $N_x=3600$, $h_n=0.005$, $M_n=400$, $M=500$.}\label{Ex6}
\end{figure}

As is well known, the Allen-Cahn equation is the $L^2$-gradient flow of the energy functional  
\begin{equation}\label{energy}
E(u) = \int_\Omega \Big( \frac{1}{2} |\nabla u|^2 + F(u) \Big) \, \d\bx,
\end{equation}
where the potential is given by $F(u) = \frac{1}{4}(u^2 - 1)^2$.  The evolution of $E(u)$ is shown in Fig.~\ref{fig:EX6_energy}. As anticipated, the energy decays monotonically in time, reflecting the gradient-flow structure of the equation in the $L^2$-framework. The sharp drop at early times corresponds to the initial coarsening stage, during which small‐scale features are rapidly eliminated and the total energy decreases significantly. This is followed by a slower decay, signalling the approach to a metastable or steady configuration. The smooth, monotonic decrease in $E(u)$ not only confirms the dissipative character of the model but also attests to the accuracy and stability of the proposed numerical scheme.

\begin{figure}[ht!]
\centering  
\subfigure{
\includegraphics[width=0.5\textwidth]{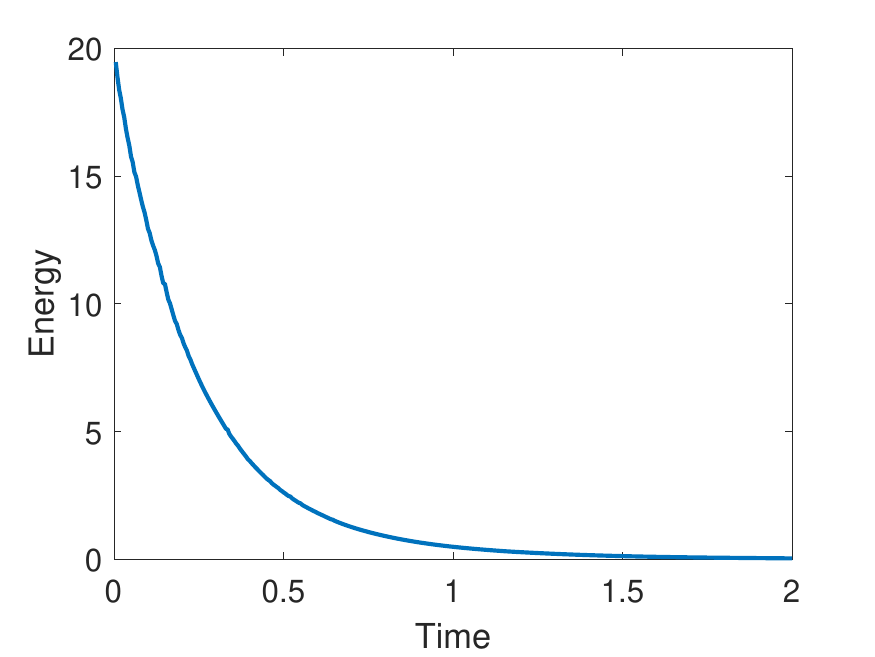}}
\caption{Time histories of the energy of Allen-Cahn equation computed by $N_x=3600$, $h_n\equiv0.005$, $M_n\equiv400$, $M=500$.}
\label{fig:EX6_energy}
\end{figure}

\underline{\textbf{(II) Phase separation from random initial data in hexagonal domain.}} %Next, we \\ take
\\Next, we investigate the evolution of interfaces between two phases starting from random initial data on a hexagonal domain defined in polar coordinates by
\begin{equation*}
R(\theta)
=
\frac{\frac{\sqrt{3}}{2} R_0}
{\cos\!\left( (\theta \bmod \frac{\pi}{3}) - \frac{\pi}{6} \right)},
\qquad \theta \in [0,2\pi),
\end{equation*}
where $R_0$ denotes the circumradius of the regular hexagon, and the initial phase field is taken as a zero-mean random field uniformly distributed in the interval $[-0.5,\,0.5]$.

As shown in Fig.~\ref{Ex6_2}, the phase field evolves from an initially unstructured state generated by uniform random inputs to a configuration exhibiting clear phase separation, with distinct phase regions forming over time. Overall, the numerical evidence suggests that the proposed approach offers a reliable and efficient tool for solving semilinear parabolic PDEs on complex computational domains.

\begin{figure}[htbp]
\centering \hspace{-16pt}
\subfigure[T=0]{\label{Ex6_2_t0} 
\includegraphics[width=4.2cm,height=4.0cm]{ 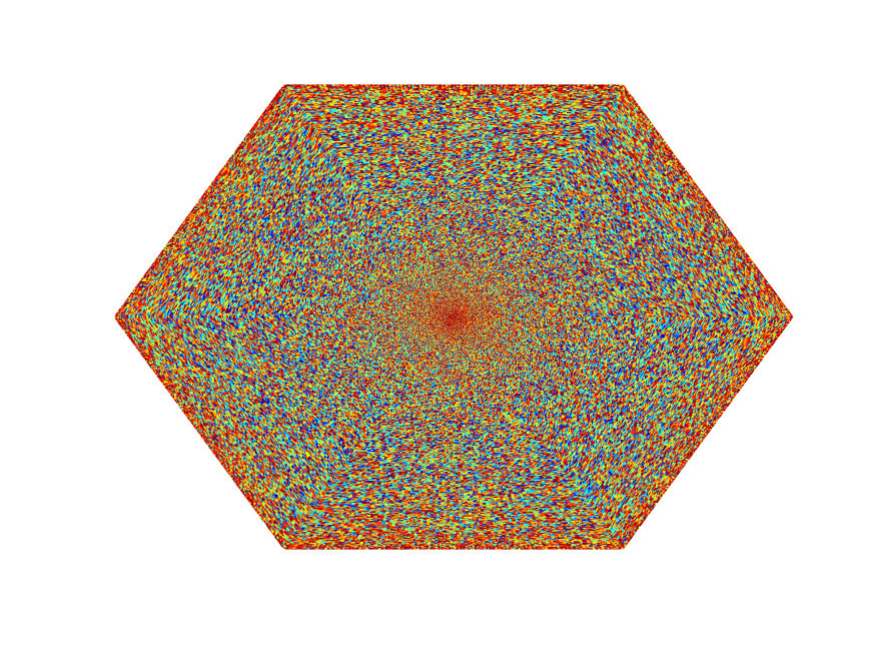}}\hspace{-16pt}
\subfigure[T=0.005]{\label{Ex6_2_t0005} 
\includegraphics[width=4.2cm,height=4.0cm]{ 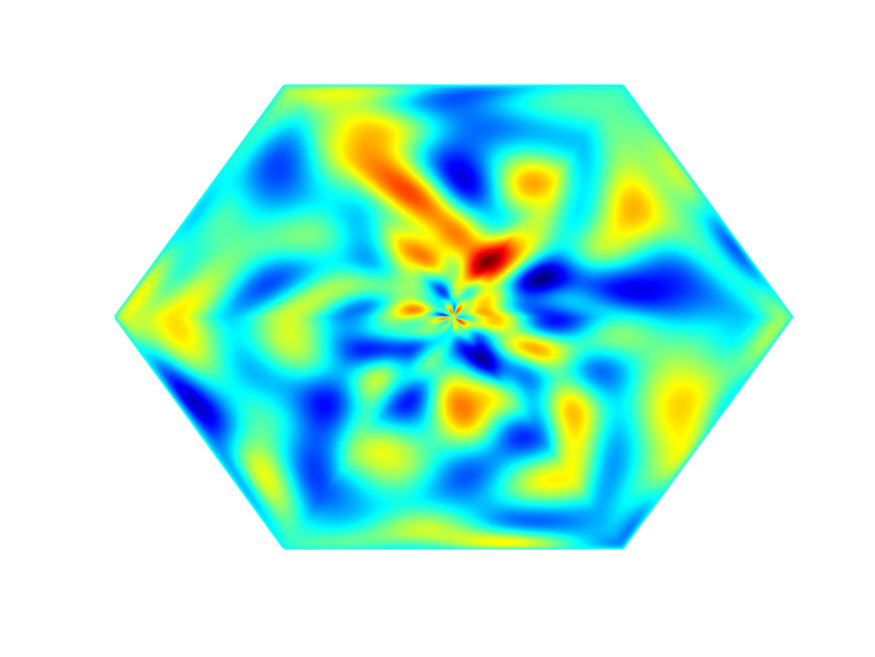}}\hspace{-16pt}
\subfigure[T=0.01]{\label{Ex6_2_t001} 
\includegraphics[width=4.2cm,height=4.0cm]{ 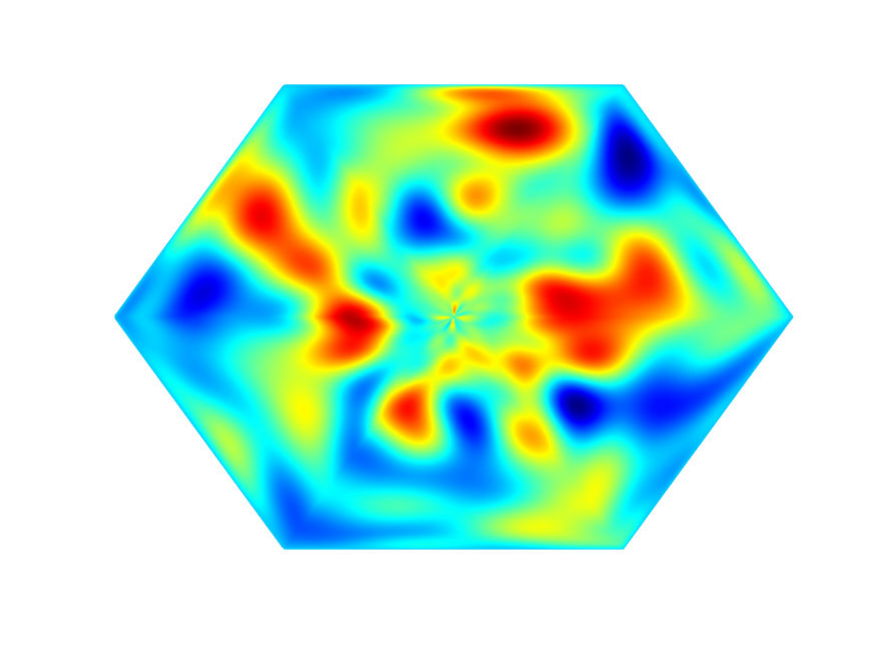}}\hspace{-16pt}\vspace{-11pt}
\subfigure[T=0.02]{\label{Ex6_2_t002} 
\includegraphics[width=4.2cm,height=4.0cm]{ 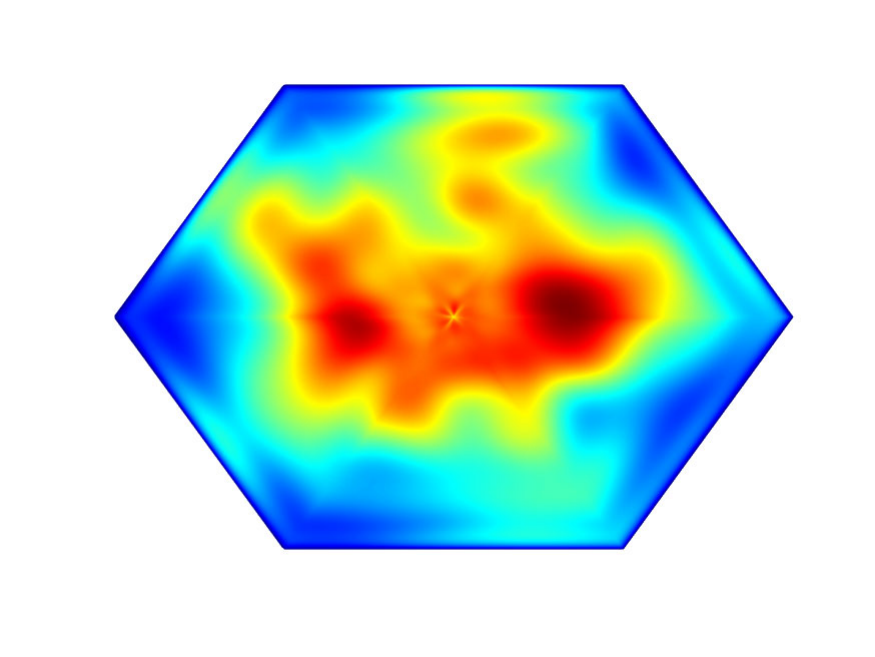}}\hspace{-16pt}
\subfigure[T=0.1]{\label{Ex6_2_t01}
\includegraphics[width=4.0cm,height=4.0cm]{ 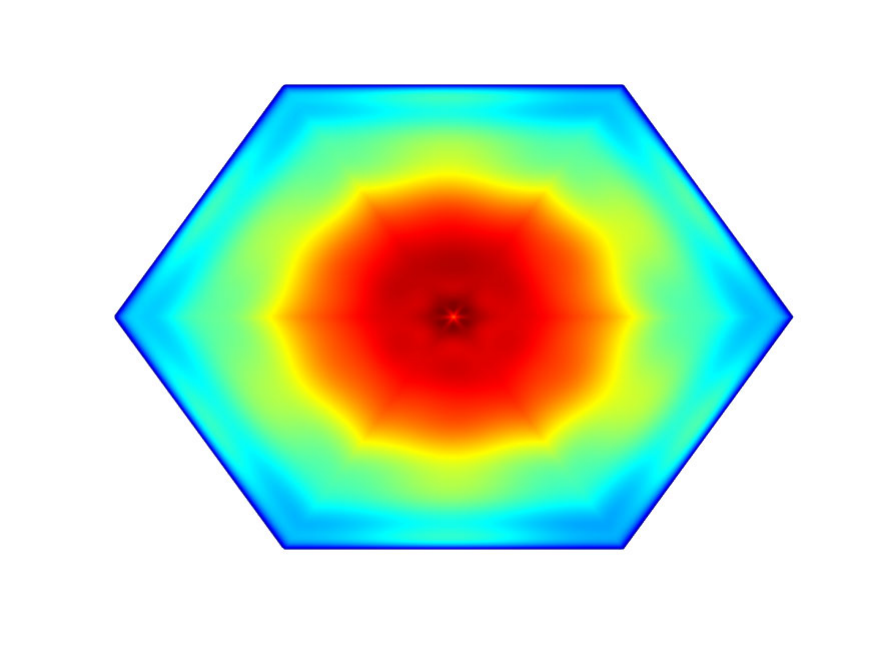}}\hspace{-16pt}
\subfigure[T=0.5]{\label{Ex6_2_t05}
\includegraphics[width=4.2cm,height=4.0cm]{ 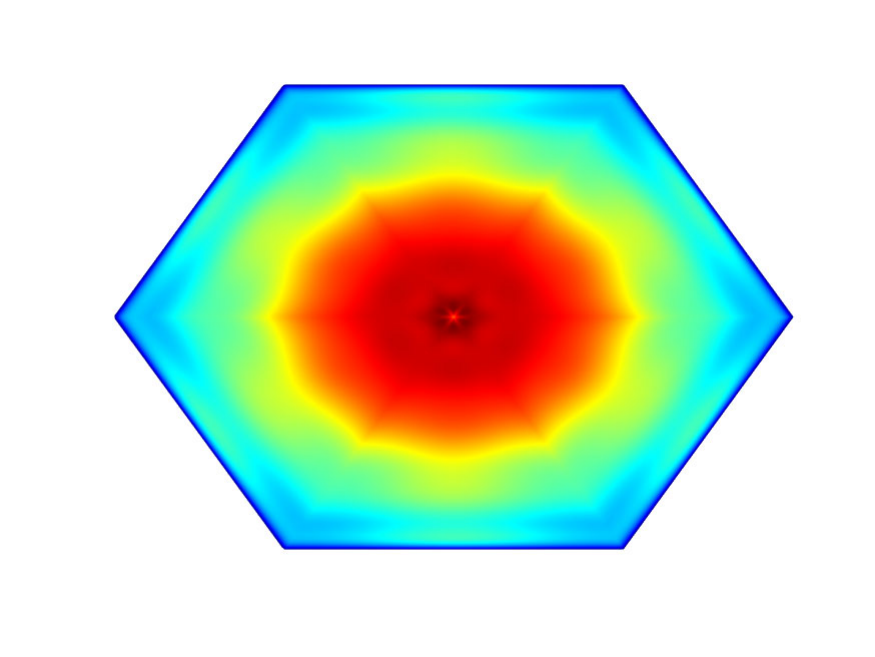}}
\caption{Numerical evolutions of Allen-Cahn equations with random input at different T. Simulations are obtained by Algorithm \ref{algo_2} with $N_x=250000$, $h_n=0.005$, $M_n=100$, $M=500$.}\label{Ex6_2}
\end{figure}
\section{Conclusions}
This work presented an $hp$-version time-stepping spectral\\ Monte Carlo method for the numerical approximation of semi-linear parabolic equations.  
The core contribution lay in the seamless integration of residual iterative schemes, spatio-temporal spectral interpolation techniques, the $hp$-version time\\-stepping framework, and stochastic sampling strategies.  
Consequently, the numerical experiments demonstrated that the proposed approach effectively overcame several intrinsic limitations of traditional stochastic methods (e.g., WoS method), including  
\begin{itemize}
\item insufficient accuracy, typically restricted to half-order convergence;
\item  the inability to perform robust long-time simulations;
\item  the difficulty of achieving high precision for solutions exhibiting initial singularities.  
\end{itemize}
 Furthermore, rigorous theoretical analysis established that the method, together with its multi-step extension, achieved exponential convergence rates.  
It is particularly noteworthy that, unlike conventional high-order deterministic space-time methods that require solving fully coupled nonlinear systems in both space and time, the proposed stochastic framework completely eliminated such coupled nonlinear solves, while enabling fully parallelized computations across spatial and temporal grids.  
Compared with existing techniques, the method exhibited substantially broader applicability.  
In particular, it provided a unified treatment of both the classical case $\alpha = 2$ and the fractional case $\alpha \in (0,2)$ within a single computational framework.  
Extensive numerical validations, including long-time simulations, spectral-accuracy recovery for solutions with initial singularities, five-dimensional benchmark problems, and fitting problems on complex domains, demonstrated the computational feasibility, scalability, and efficiency of the proposed approach for a wide class of nonlinear problems.  
Moreover, the techniques developed in this study were not restricted to stochastic algorithms; their inherent scalability suggested significant potential for applications in machine learning, which will serve as an important direction for future research.

 Finally, we demonstrate how to handle moderately high-dimensional cases. 
The key strategy is to employ spectral sparse grid techniques based on 
 \emph{nested} Chebyshev--Gauss--Lobatto nodes~(cf. \cite{ShenYu2010}), which are constructed via a combination of tensor-product subgrids. Within the spectral Monte Carlo framework, the two crucial steps are as follows:
\begin{itemize}
  \item[(i)] computing the solution values at sparse-grid nodes by the stochastic algorithms;
  \item[(ii)] reconstructing the solution and the required spatial and temporal derivatives via spectral sparse grid approximations.
\end{itemize}

To this end, we introduce a spectral sparse grid 
$\{\boldsymbol{x}^{\mathrm{sg}}_j\}_{j\in\mathcal{K}_d}$, 
where $\mathcal{K}_d$ denotes the associated sparse-grid index set, 
$\boldsymbol{x}^{\mathrm{sg}}_j=(x_{j_1},\ldots,x_{j_d})$, 
and each $\{x_{j_\ell}\}$ is given by one-dimensional Chebyshev--Gauss--Lobatto nodes.
Then the second step above can be realized through the following procedure:
\begin{equation*}
\big\{ u(\boldsymbol{x}_j^{\mathrm{sg}}, t_{n,k}) \big\}\xrightarrow[\substack{\text{nodal $\to$ modal}}]{\ \text{FFT}\ }
\big\{ \tilde u_{\boldsymbol{p},q},\tilde u^x_{\boldsymbol{p},q},\tilde u^t_{\boldsymbol{p},q} \big\}\xrightarrow[\substack{\text{modal $\to$ nodal}}]{\\\text{FFT}\ }
\big\{ u(\boldsymbol{x}, t) \big\},
\end{equation*}
where $\tilde u_{\boldsymbol{p},q}$, $\tilde u^x_{\boldsymbol{p},q}$, and $\tilde u^t_{\boldsymbol{p},q}$ denote the coefficients associated with $u(\boldsymbol{x}, t)$, $-\Delta u(\boldsymbol{x}, t)$, and $\partial_t u(\boldsymbol{x}, t)$, respectively.  Note that all these procedures can be efficiently carried out using fast transform techniques. Moreover, by incorporating suitable mappings (see, e.g.,~\cite{Wang2023}), the proposed approach can be extended to moderately high-dimensional problems on complex geometries. A detailed investigation of these topics will be reported in future work. 
\vspace{18pt}

\noindent{\bf Declarations}
\begin{itemize}
\item {\bf Availability of data and materials:}  The datasets generated during and/or analysed during the current study are available from the corresponding author on reasonable request. %The data that support the findings of this study are available from the corresponding author upon reasonable request.
\item {\bf Authors' contributions:} All authors contributed to this study. The computations and the first draft were prepared by the first and second authors. All authors read and approved the final manuscript.
\item {\bf Conflict of interest statement:}   We have no conflicts of interest to disclose.
\end{itemize}

% \bibliographystyle{siamplain}
% \bibliography{references}

\end{document}